\documentclass[11pt]{article}
\usepackage{amsfonts,amssymb,amsmath,amsthm,bm,MnSymbol}
\usepackage{latexsym,url,color,enumerate}
\usepackage{kbordermatrix}
\usepackage{array,amstext}

\newtheorem{theorem}{Theorem}[section]
\newtheorem{proposition}[theorem]{Proposition}
\newtheorem{lemma}[theorem]{Lemma}
\newtheorem{corollary}[theorem]{Corollary}

\theoremstyle{definition}
\newtheorem{definition}[theorem]{Definition}

\theoremstyle{remark}
\newtheorem{remark}[theorem]{Remark}
\newcommand{\RNum}[1]{\uppercase\expandafter{\romannumeral #1\relax}}

\newcommand{\ZZ}{\mathbb{Z}}
\newcommand{\NN}{\mathbb{N}}
\newcommand{\QQ}{\mathbb{Q}}
\newcommand{\RR}{\mathbb{R}}
\newcommand{\CC}{\mathbb{C}}

\newcommand{\Min}{\textup{Min}}
\newcommand{\disj}{\sqcup}
\newcommand{\Blue}{\color{blue}}

\title{Low dimensional strongly perfect lattices IV: 
The dual strongly perfect lattices of dimension 16.} 
\author{Sihuang Hu \thanks{husihuang@gmail.com, Humboldt fellow supported by the AvH foundation.}, Gabriele Nebe
\thanks{nebe@math.rwth-aachen.de} }

\begin{document}

\maketitle

{\small 
	{\sc Abstract.} 
	We classify the dual strongly perfect lattices in dimension 16. 
	There are four pairs of such lattices, the famous Barnes-Wall lattice
	$\Lambda _{16}$, the extremal 5-modular lattice $N_{16}$, the odd Barnes-Wall lattice
	$O_{16}$ and its dual, and one pair of new lattices $\Gamma _{16}$ and its dual.
	The latter pair belongs to a new infinite series of dual strongly perfect lattices, 
	the sandwiched Barnes-Wall lattices, described by the authors in a previous paper. 
An updated table of all known strongly perfect lattices up to dimension 26 is available 
in the catalogue of lattices  \cite{Lattices}. 
}

\tableofcontents

\section{Introduction}

The notion of strongly perfect lattices has been introduced in  the fundamental
work \cite{Venkov} by Boris Venkov based on lecture series 
Venkov gave in Aachen, Bordeaux and Dortmund.  
Strongly perfect lattices are particularly nice examples of locally densest 
lattices, they even realize a local maximum of the sphere packing density
on the space of all periodic packings (see \cite{Schuermann}). 
Together with Boris Venkov the second author started a long term project to 
classify low dimensional strongly perfect lattices. 
The strongly perfect lattices up to dimension 9 and in dimension 11 are already
classified in \cite{Venkov}. These are all root lattices and their duals. 
In dimension 10 there are two strongly perfect lattices,
the lattice $K_{10}'$ and its dual (see \cite{dim10}) and in dimension 
12 the Coxeter-Todd lattice $K_{12}$ is the unique strongly perfect lattice
(\cite{dim12}). For all known strongly perfect lattices, with one exception 
in dimension 21, also the dual lattice is strongly perfect. 
Such lattices are called {\bf dual strongly perfect} (see Section \ref{dsp}).
They are classified in dimensions 13-15 (\cite{dim14}, \cite{dim13_15}). 
The present paper continues the classification of low-dimensional 
(dual) strongly perfect lattices by treating the very interesting 
16-dimensional case. 
In dimension 16 there are (up to similarity) six dual strongly
perfect lattices (see Theorem \ref{main}), the famous Barnes-Wall lattice $\Lambda _{16}$
realizing the maximal
known sphere packing density, the odd Barnes-Wall lattice $O_{16}$ and 
its dual, the unique extremal 5-modular lattice named $N_{16}$ in \cite{Venkov}
and two new lattices, $\Gamma _{16} $ and its dual, first described in 
\cite{HuNebe}. 

The overall strategy for the classification of dual strongly perfect lattices in 
a given dimension is already described in the introduction to \cite{dim14}. 
Let $\Lambda $ be a strongly perfect lattice of dimension $n$ and put 
$s := s(\Lambda ) = \frac{1}{2} |\Min (\Lambda ) | \in \ZZ $ to denote half
of the kissing number of $\Lambda $ and
$$r := r(\Lambda )  =r(\Lambda ^*) = \min (\Lambda ) \min (\Lambda ^*) \in \QQ $$
 the Berg\'e-Martinet invariant of $\Lambda $.
 As $\Lambda $ is perfect, we obtain $s(\Lambda ) \geq \frac{n(n+1)}{2}$ (see \cite[Proposition 3.2.3 (2)]{Martinet}). 
 Upper bounds on the kissing number are 
 given for instance in \cite{MittelmannVallentin} leading to finitely many possibilities of the integer $s$.

 By \cite[Th\'eor\`eme 10.4]{Venkov} (see Lemma \ref{min}) we have $r(\Lambda ) \geq \frac{n+2}{3}$.
 As $r(\Lambda )$ is the product of the Hermite function evaluated at $\Lambda $ and its dual $\Lambda ^*$,
 we obtain $r \leq \gamma_n^2$, where $\gamma _n$ is the Hermite constant (see Section \ref{basic}). 
 The best known upper bounds on $\gamma _n$ are given in \cite{CoEl} so we obtain upper and lower 
 bounds for the rational number $r$. 
 To obtain a finite list of possible pairs $(r,s)$ we apply the equations \eqref{gleichungen} to a minimal
 vector $\alpha \in \Lambda ^*$. For instance $(D2)$ and $(D4) $ yield that 
 $sr/n$ and $3sr/(n(n+2))$ are integers and from $\frac{1}{12}(D4-D2) $ we obtain 
 that $\frac{sr}{12n} (\frac{3r}{n+2}-1)$ is an integer, giving only finitely many possibilities for $r$. 
 Using the general  lemmas from Section \ref{gen} additionally narrows down the possibilities.
 In particular for $n=16$ the possible values are listed in Theorem \ref{TableOfValues}. 
 So far we only used the fact that $\Lambda $ is strongly perfect. 

 The fact that also the dual lattice is strongly perfect is then used to obtain 
 bounds on the level of $\Lambda $:
 For each value of $r=r(\Lambda ) = r(\Lambda ^*)$ we now factor $r=m\cdot d$ 
such that the equations \eqref{gleichungen} allow to show that 
 rescaled to minimum $\min (\Lambda ^*) = m$, the lattice $\Lambda ^*$ is 
 even and in particular contained in its dual lattice $\Lambda $ (which is
 then of minimum $d$).
 For dual strongly perfect lattices we can use a similar argumentation
 to obtain a finite list of possibilities $(s',r )$ for
 $s' = s(\Lambda ^*)$ and in each case a factorization $r = m' \cdot d' $
 such that $\Lambda $ is even if rescaled to $\min (\Lambda ) = m'$.
 But this allows to obtain the exponent (in the latter scaling)
 $$\exp (\Lambda ^*/\Lambda )  \mbox{ divides }  \frac{m}{d'} $$
 which either allows a direct classification of all such lattices
 $\Lambda $ or at least the classification of all genera of such lattices and then the use  of modular forms to exclude
 the existence of a  theta series 
 $\theta _{\Lambda } $ of level $\frac{m}{d'} $ and weight $\frac{n}{2}$
 starting with
 $1+2sq^{m' } + \ldots $, such that its image under the Fricke involution starts
 with $1+2s'q^m + \ldots $ and both $q$-expansions have nonnegative
 integral coefficients.
 This computational technique using modular forms is described in more detail in Section \ref{ModForm}.

 {\bf Acknowledgements} Sihuang Hu is supported by a fellowship of the Humboldt foundation.

 \section{Some basic facts on lattices}\label{basic}

For a good introduction to the theory of lattices in Euclidean spaces 
in our context we refer to the book \cite{Martinet} by Jacques Martinet. 

A \emph{lattice} $\Lambda $ is the integral span of a basis 
$B:=(b_1,\ldots , b_n)$ of Euclidean $n$-space $(\RR^n,(,))$, i.e.
$$\Lambda = \{ \sum _{i=1}^n a_i b_i \mid a_i \in \ZZ \} .$$ 
The \emph{dual} lattice of $\Lambda $ is 
$$\Lambda ^* := \{ v\in \RR ^n \mid (v,\lambda ) \in \ZZ \mbox{ for all } 
\lambda \in \Lambda \},$$ 
the $\ZZ $-span of the dual basis of $B$. 
The two most important invariants of a lattice are its 
\emph{minimum} $$\min(\Lambda ):= \min \{ (\lambda ,\lambda )\mid 0\neq \lambda \in \Lambda \} $$ and its \emph{determinant} 
$$\det (\Lambda ) := \det ((b_i,b_j)_{1\leq i,j \leq n}) .$$
We clearly have $\det(\Lambda ) \det(\Lambda ^*) = 1$ 
and $\det(a\Lambda ) = a^{2n} \det(\Lambda )$ for all $a\in \RR_{>0}$. 

A lattice $\Lambda $ is called \emph{integral}, if $(\lambda,\lambda' )\in \ZZ $
for all $\lambda ,\lambda' \in \Lambda $, i.e. $\Lambda \subseteq \Lambda ^*$.
The lattice $\Lambda $ is called \emph{even}, if 
$(\lambda ,\lambda ) \in 2\ZZ $ for all $\lambda \in \Lambda $. 
Clearly even lattices are integral. For an even lattice $\Lambda $ the minimal natural number $\ell $ such that 
$\sqrt{\ell } \Lambda ^*$ is even is called the \emph{even level} of $\Lambda $.

Two $n$-dimensional lattices $\Lambda $ and $\Gamma $ are called
\emph{similar}, if there is a similarity 
$g\in $GL$_n(\RR )$, $(gx,gy) = a (x,y) $ (some $a\in \RR _{>0} $)
with $g\Lambda = \Gamma $. Similarities of norm $a=1$ are called \emph{isometries}.
For a similarity of norm $a$ we have 
$\det(g\Lambda ) = a^n \det(\Lambda )$ and $\min (g\Lambda ) = a\min(\Lambda )$, so the
 \emph{Hermite function}
\begin{align*}
\gamma:\ \mathcal{L}_n &\rightarrow \RR\\
[\Lambda ]  &\mapsto \gamma (\Lambda ) := \frac{\min(\Lambda)}{\det(\Lambda)^{1/n}}
\end{align*}
is well defined on the set of similarity classes 
$\mathcal{L}_n$  of all $n$-dimensional lattices.
The density of a lattice is a strictly monotonous function of the Hermite
function, so in particular the (local) maxima of $\gamma $ provide
the (locally) densest lattice sphere packings. 
It is well known (\cite[Theorem 3.5.4]{Martinet})  that there are 
only finitely many local maxima of the Hermite function on ${\mathcal L}_n$,
all of them are represented by rational lattices (\cite[Proposition 3.2.11]{Martinet}), i.e. 
$(\lambda, \lambda ') \in \QQ $ for all $\lambda, \lambda ' \in \Lambda $. 
In particular the 
 \emph{Hermite constant}. 
 $\gamma_n=\sup\{\gamma(\Lambda) \mid \Lambda \in \mathcal{L}_n\}  $ 
is  attained at some integral lattice.
The densest lattices (and hence $\gamma _n$) are 
known in dimension $\leq 8$ and in dimension 24 (\cite{CohnKumar}). 
The best known upper bounds on the Hermite constant are given in \cite{CoEl}. 
These also yield the best known upper bounds for the 
\emph{Berg\'e-Martinet invariant} $r(\Lambda )$, where 
$$r(\Lambda):=\gamma (\Lambda ) \gamma (\Lambda ^*) = 
\min(\Lambda)\min(\Lambda^*)$$ as $r(\Lambda)\le \gamma_n^2$. 
By the definition of the Hermite constant, we obtain the following inequalities.

\begin{lemma}\label{DeterminantBound}(\cite[Lemma 2.1]{dim14})
Let $\Lambda$ be an $n$-dimensional lattice. Then
\begin{align*}
\left(\frac{\gamma_n}{\min(\Lambda^*)}\right)^n 
\ge \det(\Lambda) 
\ge \left(\frac{\min(\Lambda)}{\gamma_n}\right)^n.
\end{align*}
\end{lemma}

\begin{lemma}\label{IsLevelIntegral}(\cite[Lemma 2.1.12]{Nossek})
    Let $\Lambda$ be an integral lattice in dimension $n$. If there exists some rational number $c$
    such that $\sqrt{c}\Lambda^*$ is integral, then $c$ is an integer.
\end{lemma}
\begin{proof}
As $\det(\Lambda)\cdot\det(\sqrt{c}\Lambda^*)=c^n$ is an integer, the number $c$ is an integer.
\end{proof}

\section{Strongly perfect lattices} \label{gen}

For a lattice $\Lambda $ and some $a\in \RR $ we put 
$$\Lambda _a:=\{ \lambda \in \Lambda \mid (\lambda,\lambda ) = a \}.$$
This is always a finite set  invariant under multiplication by $-1$.
Of particular interest is the set 
$\Lambda _m =: \Min (\Lambda) $  of minimal vectors in $\Lambda $,
where $m = \min (\Lambda ) $.

\begin{definition} \label{stperf}
	A lattice $\Lambda $ is called \emph{strongly perfect},
	if $\Min(\Lambda )$ forms a spherical 4-design.
\end{definition}

It is well known (\cite[Th\'eor\`eme 6.4]{Venkov}, \cite[Theorem 16.2.2]{Martinet}) that strongly perfect lattices 
are extreme, i.e. they realize a local maximum of the Hermite function on the space of similarity classes of 
$n$-dimensional lattices. In particular strongly perfect lattices are always similar to rational lattices. 

We usually write $\Min (\Lambda ) = S(\Lambda) \disj -S(\Lambda) $ as a disjoint union 
and call $s:=s(\Lambda ):= |S(\Lambda)| $ the \emph{half kissing number} of $\Lambda $. 
By \cite[Th\'eor\`eme 3.2, Equation (5.2b)]{Venkov} the lattice $\Lambda $ is strongly perfect, 
if and only if 
\begin{equation}\label{D4}
	(D4)(\alpha ): \ \ \sum _{x\in S(\Lambda)} (x,\alpha )^4 = \frac{3 s(\Lambda )}{n(n+2)} 
\min (\Lambda ) ^2 (\alpha,\alpha )^2 
\end{equation}
for all $\alpha \in \RR ^n$.

From $(D4)(\alpha )$ we obtain the following
equations $(Di) = (Di)(\alpha )$ and $(Dij) = (Dij)(\alpha,\beta) $ for all $\alpha , \beta \in \RR ^n$:
\begin{equation}\label{gleichungen}
\begin{array}{lrr}
	(D2)(\alpha): &   \sum _{x\in S(\Lambda)} (x,\alpha )^2 = &  \frac{sm}{n} (\alpha , \alpha )  \\
	(D11)(\alpha , \beta): &   \sum _{x\in S(\Lambda)} (x,\alpha )(x,\beta ) = &  \frac{sm}{n} (\alpha , \beta )  \\
	(D22)(\alpha , \beta): &  \sum _{x\in S(\Lambda)} (x,\alpha )^2(x,\beta)^2 = & 
	\frac{sm^2}{n(n+2)}
	(2(\alpha , \beta )^2+(\alpha  , \alpha )( \beta , \beta )) \\
	(D13)(\alpha , \beta): &   \sum _{x\in S(\Lambda)} (x,\alpha  ) (x,\beta )^3= &  \frac{3sm^2}{n(n+2)} (\alpha , \beta )(\beta , \beta )\\
	\frac{1}{12}(D4-D2)(\alpha) : &  \frac{1}{12} \sum _{x\in S(\Lambda)} (x,\alpha )^4 - (x,\alpha )^2 = &
	\frac{sm}{12n} (\alpha , \alpha ) (\frac{3m}{n+2} (\alpha,\alpha ) - 1 ) 
\end{array} 
\end{equation}

Note that $(D2)(\alpha )$, $(D22)(\alpha, \beta )$,
$(D4)(\alpha ) $, $\frac{1}{12} (D4-D2)(\alpha)$ 
are non negative integers for all $\alpha ,\beta \in \Lambda ^*$.
In particular for $\alpha \in \Min (\Lambda ^*)$ we obtain
$$\frac{1}{12} (D4-D2)(\alpha )  = 
\frac{s(\Lambda )}{12 n} r(\Lambda ) (\frac{3}{n+2} r(\Lambda ) - 1 ) \in \ZZ  _{\geq 0 }$$
whence

\begin{lemma}{\label{min}}(\cite[Th\'eor\`eme 10.4]{Venkov})
    Let $\Lambda$  be a strongly perfect lattice of dimension $n$.
    Then the Berg\'e-Martinet invariant
    $$r(\Lambda)\geq \frac{n+2}{3}.$$
\end{lemma}
A strongly perfect lattice $\Lambda$ is called {\em of minimal type} if the above equality holds,
and \emph{of general type} otherwise.
Let $\Lambda$ be a strongly perfect lattice of dimension $n$. Set $m=\min(\Lambda)$ and
$s=s(\Lambda)=|S(\Lambda )|$.

\begin{lemma}\label{linkombext}
	Let $\alpha \in \RR ^n$ be such that $(x, \alpha) \in \ZZ$ for all $x\in S(\Lambda )$.
	Denote $\ell =\max\{(x,\alpha): x \in \Min(\Lambda )\}$. 
	Let $N_i(\alpha ) = \{ x\in \Min(\Lambda )  \mid (x,\alpha ) = i \}$ for $i=1,\ldots,\ell$,
    and let $$c=\frac{sm}{6n}\left(\frac{3m}{n+2}(\alpha ,\alpha )  -1 \right) .$$
    Then
    \begin{equation} \label{eins}
\sum_{i=2}^{\ell}\,\sum_{x\in N_{i}(\alpha)}\frac{i(i^2-1)}{6}x=c\alpha
\end{equation}
    and
    \begin{equation} \label{zwei}
\sum_{i=2}^{\ell}\frac{i^2(i^2-1)}{6}|N_{i}(\alpha)|=c(\alpha,\alpha).
\end{equation}
\end{lemma}

\begin{proof}
	By \eqref{gleichungen} we obtain 
	\begin{equation} \label{D13-D11} 
		\frac{1}{6} (D13-D11)(\beta,\alpha) : \ \ 		\frac{1}{6} 
		\sum_{x\in S(\Lambda )} ((x,\alpha)^3 (x,\beta ) - (x,\alpha)(x,\beta) ) 
		= c (\alpha,\beta )
	\end{equation} 
	where $c$ and $\alpha $ are  as in the lemma and $\beta \in \RR ^n$ 
	is an arbitrary vector. 
	Equation \eqref{D13-D11} is easily seen to be the inner
	product 
 of Equation \eqref{eins} with $\beta $. As $\beta $ is arbitrary, we obtain 
 Equation \eqref{eins}. 
 Equation \eqref{zwei} is obtained by taking the inner product of 
 Equation \eqref{eins} with $\alpha $.
\end{proof}

\begin{corollary}{\label{linkomb}}(\cite[Lemma 2.1]{dim10})
    Let $\alpha \in \RR ^n$ be 
    such that $(x, \alpha) \in \{ 0,\pm 1 ,\pm 2 \} $ for all $x\in \Min (\Lambda )$.
    Let $N_2(\alpha ) = \{ x\in \Min(\Lambda ) \mid (x,\alpha ) = 2 \}$ and
    put $$c=\frac{sm}{6n}\left(\frac{3m}{n+2}(\alpha ,\alpha )  -1 \right) .$$
    Then 
    $|N_2(\alpha ) | = c (\alpha , \alpha ) /2 $ and
    $$\sum _{x\in N_2(\alpha )} x = c \alpha . $$
\end{corollary}

\begin{lemma} \label{N_2(alpha)neq1} (\cite[Lemma 2.6]{dim12})
  Let $\Lambda$ be a strongly perfect lattice and choose $\alpha\in\Min(\Lambda^*)$ that satisfies the conditions
  of Corollary~\ref{linkomb}. If $n\ge 11$ then $|N_2(\alpha)|\ne 1$.
\end{lemma}



\begin{lemma}{\label{boundn2r<8V1}}(\cite[Lemma 2.4]{dim12},\cite[Lemma 2.7.18]{Nossek})
    Suppose $\alpha\in\Min(\Lambda^*)$. If $r(\Lambda) < 8$, then
    $$|N_2(\alpha )| \leq \min\left\{\frac{r(\Lambda)}{8-r(\Lambda)}, n\right\}.$$
    The equality $|N_2(\alpha)|=\frac{r(\Lambda)}{8-r(\Lambda)}$ 
    holds if and only if $N_2(\alpha)$ spans a rescaled root lattice $A_{|N_2(\alpha)|}$.
\end{lemma}


\begin{definition}
Let $A$ be a subset of the interval $[-1,1)$. A \emph{spherical $A$-code}
is a non-empty subset $X$ of the unit sphere in $\RR^n$, satisfying that $(x,y)\in A$, 
for all $x\ne y\in X$.
\end{definition}

\begin{lemma}\label{SphCdeUppBnd}(\cite[Example 4.6]{DGS})
For a given number $a$, with $0 \le a < n^{-1/2}$, let $A$ be any subset of $[-1,a]$,
and let $X$ be a spherical $A$-code in $\RR^n$. Then
$$
|X| \le \frac{n(1-a)(2+(n+1)a)}{1-na^2}.
$$
\end{lemma}

\begin{lemma}\label{SphBndN_2}
Let $\Lambda$ be a strongly perfect lattice of dimension $n$ with $r(\Lambda)\ge 8$. Let $\alpha\in\Min(\Lambda^*)$,
and $N_2(\alpha)=\{x \in \Min (\Lambda ) \mid (x,\alpha)=2\}.$ Denote $a=(r(\Lambda)-8)/(2(\Lambda)r-8)$.
If $a< (n-1)^{-1/2}$, then
$$
|N_2(\alpha)| \le \frac{(n-1)(1-a)(2+na)}{1-(n-1)a^2}.
$$
\end{lemma}
\begin{proof}
	Without loss of generality, we rescale $\Lambda$ such that $\min(\Lambda)=1$, and $\min(\Lambda^*)=r(\Lambda) =:r$.
Define
$$
\overline{N_2}(\alpha)=\{ \sqrt{\frac{r}{r-4}}(x-\frac{2}{r}\alpha) \mid x\in N_2(\alpha)\}.
$$
Then $|N_2(\alpha)|=|\overline{N_2}(\alpha)|$, and 
for any two elements $\bar{x},\bar{y}\in\overline{N_2}(\alpha)$, we have $(\bar{x},\alpha)=0$, and
$$
(\bar{x},\bar{y})=\frac{r}{r-4}(x-\frac{2}{r}\alpha, y-\frac{2}{r}\alpha)
\begin{cases}
= 1 & \text{if } \bar{x}=\bar{y}\\
\le \frac{r-8}{2r-8} & \text{if } \bar{x}\ne\bar{y}.
\end{cases}
$$
Hence $\overline{N_2}(\alpha)$ is a spherical $[-1,\frac{r-8}{2r-8}]$-code in $\RR^{n-1}$,
now the assertion follows from Lemma~\ref{SphCdeUppBnd} directly.
\end{proof}

\begin{corollary}\label{boundr8}(\cite[Lemma 2.8]{dim14})
	If $r(\Lambda)=8$ and $\alpha \in \Min(\Lambda )$ then 
    $$|N_2(\alpha ) |\leq 2(n-1) .$$
    If equality holds then the sublattice of $\Lambda $ generated by 
    $N_2(\alpha )$ is similar to the root lattice $D_n$.
\end{corollary}

We now apply the above equations to obtain a finite list of pairs $(r(\Lambda ), s(\Lambda ))$  
for dimension $n=16$.

\begin{theorem}\label{TableOfValues}
Let $\Lambda$ be a strongly perfect lattice of dimension 16. Then for $r(\Lambda)$ and
$s(\Lambda)$ only the values in the following table occur or $\Lambda$ is of minimal type, i.e. $r(\Lambda ) = 6$.
\begin{center}
\begin{tabular}{|>{$}c<{$}|>{$}l<{$}|>{$}l<{$}|>{$}l<{$}|>{$}l<{$}|>{$}l<{$}|>{$}l<{$}|>{$}l<{$}|>{$}l<{$}|>{$}l<{$}|>{$}l<{$}|}
  \cline{1-9}
r & 192/31    & 144/23 & 32/5      & 72/11 & 192/29 & 20/3 & 48/7      & 7   \\ 
s & 961\cdot a& 2116   & 450\cdot a& 968   & 841    & 1296 & 196\cdot a& 1152 \\ 
a & 2,3       & -      & 2,3,4     & -     & -      & -    & 2,\dots,6 & -    \\ 
\hline
r & 64/9 & 36/5      & 22/3 & 96/13     &15/2      & 144/19 & 192/25     &54/7 \\
s & 729  & 400\cdot a& 1296 & 338\cdot a&512\cdot a& 1444   & 625\cdot a &784  \\ 
a & -    & 1,2,3     & -    & 1,\dots,4 &1,2,3     & -      & 1,2        & -   \\
\hline
r  &8         & 384/47 & 90/11     & 33/4 & 192/23    & 42/5   & 144/17     & 128/15     \\
s  &72\cdot a & 2209   & 968\cdot a& 2048 & 529\cdot a& 400\cdot a & 1156\cdot a& 2025   \\
a  &2,\dots,30& -      & 1,2       & -    & 1,\dots,4 & 1,\dots,5  & 1,2        & -  \\
\hline
r   &60/7      &26/3      &96/11     & 150/17 & 384/43 & 9         & 64/7      \\
s   &784\cdot a&648\cdot a&242\cdot a& 2312   & 1849   & 128\cdot a& 441\cdot a\\
a   &1,2       &1,2,3     &1,\dots,9 & -      & -      & 2,\dots,26& 1,\dots,8 \\
\cline{1-8}
\end{tabular}
\end{center}
\end{theorem}

\begin{proof}
In~\cite{MittelmannVallentin} Mittelmann and Vallentin computed that the kissing number in dimension $16$
is upper bounded by $7355$, so $s(\Lambda)\le 3677$. On the other hand, by the lower bound on the
cardinality of spherical-$5$ designs~\cite[Theorem 5.12]{DGS}, we have $s(\Lambda) \ge 136$.
The Cohn--Elkies bound (see~\cite[Table 3]{CoEl}) implies that the Hermite constant $\gamma_{16} \le 3.027$,
hence 
$$
6 \le r(\Lambda)=\min(\Lambda)\min(\Lambda^*) \le \gamma_{16}^2 \le 9.162729.
$$
Now we compute all solutions of
$$ 6|N_3(\alpha)|+|N_2(\alpha)| = \frac{s(\Lambda)r(\Lambda)}{12\cdot16}\left(\frac{r(\Lambda)}{6}-1\right) $$
where $6|N_3(\alpha)|+|N_2(\alpha)|$ is integral and $r(\Lambda)$ is rational.
The table lists all solutions that satisfy 
 Lemma~\ref{N_2(alpha)neq1}, Lemma~\ref{boundn2r<8V1},
Lemma~\ref{SphBndN_2} and Lemma~\ref{boundr8}. 
\end{proof}

\section{Maximal even lattices}
\newcommand{\Aut}{\mbox{\rm Aut}}
\newcommand{\mass}{\mbox{\rm mass}}
During the classification of strongly perfect lattices we often know that a 
strongly perfect lattice $\Gamma $ is even of a bounded even level $\ell $, and that
$\min (\Gamma ^*) \geq d$. Then $\Gamma $ is contained in a maximal even lattice
$M$, $$\Gamma \subseteq M \subseteq M^* \subseteq \Gamma ^*$$ 
such that the even level of $M$ divides $\ell $ and $\min (M^*) \geq \min (\Gamma ^*) \geq d$. 
Therefore it is helpfull to know all such maximal even lattices $M$.
Then we may construct the lattice $\Gamma $ as a sublattice of $M$.

The set of all maximal lattices can be partitioned into genera, 
where two lattices belong to the same genus, if they are isometric locally everywhere.
Any genus 
 consists of finitely many isometry classes the number of which is called the 
 class number of the genus.
To find all maximal lattices of a given determinant we first list all possible 
genera and then construct all lattices in the genus 
using the Kneser neighbouring method
\cite{Kneser} (see also \cite{ScharlauHemkemeier}).
To check completeness we additionally compute the 
mass of the genus and use the mass formula.

\begin{proposition}\label{maxeven} 
The following table lists all genera of  maximal even lattices $M$ such that
$\det M=2^a3^b$ for some nonnegative integers $a$ and $b$.
The first column gives the genus symbol as explained in \cite[Chapter 15]{SPLAG}, followed by the class number $h$.
Then we give one representative of the genus which is usually
a root lattice, in which $\perp $ denotes the orthogonal sum.
The last column gives the mass of the genus.
\end{proposition}
$$
\begin{array}{|c|c|c|c|c|}
	\hline
	\mbox{ genus } & \mbox{ level } & h & \mbox{ repr. } & \mbox{ mass } \\
	\hline
	\textup{II}_{16}  & 1 & 2  & E_8 \perp E_8      &  691/(2^{30}3^{10}5^{4}7^{2}\cdot 11\cdot 13)\\\hline
	\textup{II}_{16}(2^{-1}_34^1_1)  & 8 & 14 & E_7 \perp D_{9}    &  691\cdot 24611/(2^{27}3^{9}5^{3}7^{2}\cdot 11\cdot 13)\\\hline 
	\textup{II}_{16}(2^{-2}_23^{-1})  & 12& 17 & A_2 \perp D_{14}   &  691\cdot 1801/(2^{27}3^{9}5^{3}7^{2})\\\hline
	\textup{II}_{16}(2^{2}_23^1)  & 12& 19 & E_6 \perp D_{10}   &  691\cdot 1801/(2^{27}3^{9}5^{3}7^{2})\\\hline
	\textup{II}_{16}(2^{-1}_34^1_73^{-1})  & 24& 60 & A_2 \perp E_7 \perp D_7 & 73\cdot 193\cdot 691\cdot 1103/(2^{27}3^{8}5^{3}7^{2}\cdot 11\cdot 13)\\\hline
	\textup{II}_{16}(2^{-1}_34^{-1}_33^1) & 24& 57 & E_6 \perp E_7 \perp D_3 & 73\cdot 193\cdot 691\cdot 1103/(2^{27}3^{8}5^{3}7^{2}\cdot 11\cdot 13)\\\hline
	\textup{II}_{16}(2^{-2}_{II}3^2) & 6 & 45 & A_2 \perp A_2 \perp D_{12} & 17\cdot 41\cdot 127\cdot 691\cdot 1093/(2^{28}3^{10}5^{2}7^{2}\cdot11\cdot13)\\\hline
	\textup{II}_{16}(2^{-1}_34^{-1}_53^2)  & 24& 294& A_2 \perp A_2 \perp E_7 \perp D_5 & 17\cdot 193\cdot 547\cdot691\cdot14611/(2^{27}3^{9}5^{3}7^{2}\cdot11\cdot13)\\\hline
\end{array} 
$$

\begin{proof}
	Let $M$ be a maximal even lattice.
	Then $$q:M^*/M \to \QQ/\ZZ , q(x+M) := \frac{1}{2} (x,x) + \ZZ $$ defines an anisotropic 
	quadratic form on the discriminant group. 
	Clearly $(M^*/M , q)$ is the orthogonal sum of its Sylow $p$-subgroups. 
	For $p>2$ the Sylow $p$-subgroup is elementary abelian of order $1,p,$ or $p^2$ (see \cite[Section 5.1]{Scharlau}).
	For $p=2$ \cite[Lemma 2.5]{dim10} lists the orthogonal summands of anisotropic $2$-groups,
	from which we conclude that the order of the Sylow 2-subgroup of $M^*/M$ is bounded by $8$. 
	So we are left to enumerate all genus symbols of 16-dimensional even lattices of determinant
	dividing $72$, construct one lattice in each genus, check maximality and then 
	compute representatives for all isometry classes in the genus with  the Kneser neighbouring method.
\end{proof}

\begin{lemma}\label{minge3/2}
	  Let $\Lambda$ be a strongly perfect even lattice of dimension $16$.
	    If $\det(\Lambda)=2^a3^b$ for some nonnegative integers $a,b$ and $\min(\Lambda^*)\ge 3/2$,
	      then $\Lambda$ is similar to one of $\Lambda_{16},\Gamma _{16},$ or $O_{16}^*$ 
	      as given in Theorem \ref{main}.
      \end{lemma}

      \begin{proof}
	     Starting with the lattices $M$ from Proposition \ref{maxeven} we successively construct 
	     sublattices $L$ of index 2 and 3 such that $\min(L^*) \ge 3/2$. 
	     The total number of isometry classes of such lattices is 63, only three of them are strongly perfect. 
      \end{proof}

      \begin{lemma}\label{leveldiv6}
	        Let $\Lambda$ be a strongly perfect even lattice of dimension $16$.
		  If the even level of $\Lambda$ divides $6$ and $\min(\Lambda^*)\ge 1$, then $\Lambda\cong\Lambda_{16}$.
	  \end{lemma}

	  \begin{proof}
		  As in the proof of Lemma \ref{minge3/2} we start with the maximal even lattices 
		  and
		  sucessively compute sublattices $L$ of even level dividing $6$ with $\min (L^*) \ge 1$. 
		    There are in total $49552$ isometry classes of such lattices.
 Among those lattices there is only one strongly perfect lattice $\Lambda_{16}$.
	  \end{proof}

  \section{Dual strongly perfect lattices}\label{dsp}
A lattice $\Lambda\subset\RR^n$ is called \emph{dual strongly perfect} if both $\Lambda$ and its dual $\Lambda^*$
are strongly perfect. As both lattices $\Lambda $ and $\Lambda ^*$ are extreme and the characterization of 
dual extreme lattices in \cite[Section 10.5]{Martinet} allows to deduce that dual strongly perfect lattices
realize a local maximum of the Berg\'e-Martinet invariant  $r(\Lambda ) = \min(\Lambda ) \min (\Lambda ^*)$ on
the space of similarity classes of $n$-dimensional lattices. 

The aim of the rest of this paper is to prove the following main result.
\begin{theorem}\label{main}
	Let $(\Lambda , \Lambda ^*)$ be a pair of dual strongly perfect lattices in dimension 16.
	Then, up to similarity and interchanging $\Lambda $ and $\Lambda ^*$, the lattices are 
	as given in the following table.
\end{theorem}
\begin{center}
\begin{tabular}{|c|c|c|c|c|c|} 
		\hline
	name & m & d & s & t & smith \\ 
		\hline
	$\Lambda _{16}$ & 4 & 2 & 2160 & 2160 & $2^8$ \\ 
	$N _{16}$ & 6 & 6/5 & 1200 & 1200 & $5^8$ \\ 
		$O_{16}$ & 3 & 2 & 256 & 1008 & $2^6$ \\ 
		$\Gamma _{16}$ &  4 & 3/2 & 432 & 768 & $2^84^2$ \\
		\hline
	\end{tabular}
\end{center}
	The first column gives the name of the lattice $\Lambda $, rescaled such that 
	$\Lambda $ is integral and primitive. The lattices in the
	first three rows are already in \cite[Table 19.1]{Venkov}. 
	The lattice $\Gamma _{16}$ is a sublattice of $\Lambda _{16}$ and described as $\Gamma _{\{ 2\}}$ in
	\cite[Section 9]{HuNebe}.
	The other columns give $m= \min (\Lambda )$, $d=\min (\Lambda ^*)$, $s= s(\Lambda )$ and $t=s(\Lambda ^*)$. 
	The last column displays the Smith invariant of the finite abelian group $\Lambda ^*/\Lambda $.

Let $\Lambda $ be a dual strongly perfect lattice. 
Clearly  $r(\Lambda ) = r(\Lambda ^*)$ and for both lattices we are hence in the 
same of the 32 cases listed in Theorem \ref{TableOfValues}. 

A purely computational argument allowing to exclude quite a few cases from 
Theorem \ref{TableOfValues} is provided by the following result proved in the thesis of Elisabeth Nossek.

\begin{lemma}\label{N_2N_2Dual}(\cite[Lemma 2.7.20]{Nossek})
  Let $\Lambda$ be a dual strongly perfect lattice of dimension $n$.
  Put $r=r(\Lambda)=r(\Lambda^*), s=s(\Lambda) $, and $t=s(\Lambda ^*)$.
  Then $$\frac{s\cdot t\cdot r}{(6n)^2}\left(\frac{3r}{n+2}-1\right)^2\in\ZZ.$$
\end{lemma}
\begin{proof}
  Rescale $\Lambda$ such that $\min(\Lambda)=1$ and $\min(\Lambda^*)=r$.
  Denote $l=\max\{(x,\alpha): x \in \Min(\Lambda), \alpha\in\Min(\Lambda^*)\}$. 
  Let $x\in\Min(\Lambda)$ and $\alpha\in\Min(\Lambda^*)$. 
  For $i=1,\ldots,l$, 
  set $N_{i,\Lambda}(\alpha ) = \{ y\in \Min(\Lambda) \mid (y,\alpha ) = i \}$,
  and $N_{i,\Lambda^*}(x)=\{\beta \in \Min(\Lambda^*) \ |\ (x,\beta) =i\}$. Let
  \begin{align*}
  c &=\frac{s}{6n}\left(\frac{3r}{n+2}-1 \right),\\
  c' &=\frac{tr}{6n}\left(\frac{3r}{n+2}-1 \right).
  \end{align*}
  By Lemma~\ref{linkombext},
  \begin{align*}
    \sum_{i=2}^{l}\,\sum_{y\in N_{i,\Lambda}(\alpha)}\frac{i(i^2-1)}{6}y&=c\alpha,\\
    \sum_{i=2}^{l}\,\sum_{\beta\in N_{i,\Lambda^*}(x)}\frac{i(i^2-1)}{6}\beta&=c'x.
  \end{align*}
  Hence
  \begin{align*}
    cc'\alpha
    = \sum_{i=2}^{l}\,\sum_{y\in N_{i,\Lambda}(\alpha)}\frac{i(i^2-1)}{6}c'y
    = \sum_{i=2}^{l}\,\sum_{y\in N_{i,\Lambda}(\alpha)}\frac{i(i^2-1)}{6}
    \sum_{j=2}^{l}\,\sum_{\beta\in N_{j,\Lambda^*}(y)}\frac{j(j^2-1)}{6}\beta
  \end{align*}
  Write $cc' =\lfloor cc' \rfloor+ \{ cc' \}$ where $0\le\{ cc' \}<1$ is the fractional part of $cc'$.
  If $cc'$ is not an integer, then $0\ne\{cc'\}\alpha\in\Lambda^*$, which contradicts the minimality of $\alpha$. Therefore
  $$ cc' = \frac{s\cdot t\cdot r}{(6n)^2}\left(\frac{3r}{n+2}-1\right)^2\in\ZZ. $$
\end{proof}

\begin{remark}\label{shortTableOfValues}
Applying Lemma \ref{N_2N_2Dual} to the values provided in Theorem \ref{TableOfValues} 
we obtain that the triple $(r(\Lambda ), s(\Lambda ) , s(\Lambda ^*) ) $ of a dual strongly 
perfect lattice in dimension 16 that is not of minimal type is as listed in the following table.
\end{remark}
\begin{center}
\begin{tabular}{|>{$}c<{$}|>{$}l<{$}|>{$}l<{$}|>{$}l<{$}|>{$}l<{$}|>{$}l<{$}|>{$}l<{$}|>{$}l<{$}|>{$}l<{$}|>{$}l<{$}|}
  \cline{1-8}
  r &    32/5 \ (\ref{32div5})      &   20/3 \ (\ref{20div3}) & 48/7 \ (\ref{lemma:level})     & 7 \ (\ref{case7})    & 36/5 \ (\ref{36div5})      & 22/3 \ (\ref{lemma:2adic_sublattice}) & 96/13 \ (\ref{lemma:level})  \\
s &     900\cdot a&      1296 & 196\cdot a& 1152 &  400\cdot a& 1296 & 676\cdot a \\
a &        1,2     &  -    & 2,3,4,6 & -    &  1,2,3     & -    & 1,2 \\
cond &  2\mid ab     &  -    & 12 \mid ab  & -    &  -     & -    & 2 \mid ab  \\
\hline
r  &15/2   \ (\ref{lemma:level})   & 54/7 \ (\ref{lemma:2adic_sublattice})  & 8 \ (\ref{case8})        &  90/11 \ (\ref{lemma:2adic_sublattice})    & 33/4 \ (\ref{lemma:level})& 192/23  \ (\ref{lemma:level})  & 42/5  \ (\ref{lemma:level})   \\
s  &512\cdot a& 784  & 72\cdot a &  968\cdot a& 2048 & 2116 & 400\cdot a   \\
a  &1,2,3       & -   & 2,\dots,30&  1,2       & -    & - & 1,\dots,5     \\
cond  & 3 \mid ab       & -  & 2 \mid ab  &  -       & -    & - & 3 \mid ab       \\
\hline
r   & 144/17 \ (\ref{144div7})  &60/7   \ (\ref{lemma:level})   &26/3   \ (\ref{lemma:2adic_sublattice})   &96/11 \   (\ref{96div11})  &  9 \ (\ref{r9})         & 64/7  \ (\ref{64div7})     \\
s    & 2312  &784\cdot a&648\cdot a&242\cdot a&  128\cdot a& 882\cdot a\\
a   & - &1,2       &1,2,3     &3,4,6,8,9 &  2,\dots,26& 1,\dots,4 \\
cond   & - &-       &-     & 24 \mid ab &  - & 4 \mid ab \\
\cline{1-7}
\end{tabular}
\end{center}
Here the line $s$ lists the possibilities for $s(\Lambda ) = $number$\cdot a$ and 
$s(\Lambda^*)=$number$\cdot b$, where the possibilities for $a$ and $b$ are given in the line headed by $a$
with respect to certain divisibility conditions deduced from Lemma \ref{N_2N_2Dual} as given in the line headed cond. 
In brackets behind the value of $r(\Lambda )$ we give the reference to where this case is dealt with in
this paper.
Applying the next lemma,  allows to exclude the first two values for $r(\Lambda )$ 
using an easy computation.

\begin{lemma}\cite[Theorem 2.9]{dim13_15}\label{lemma:polynomial_method}
Let $\Lambda$ be a dual strongly perfect lattice of dimension $n$ with $r(\Lambda)=r(\Lambda^*)=r$.
Assume that $(\alpha,x)\in\{0,\pm 1,\pm 2\}$ for all $\alpha\in\Min(\Lambda^*), x\in\Min(\Lambda)$.
Put $n_i=|\{\langle\alpha,x\rangle\in S(\Lambda^*)\times S(\Lambda) \mid (\alpha,x)=\pm i\}|$ for $i=0,1,2$. Then
\begin{align*}
  n_2 &= \frac{tsr}{12n}\left(\frac{3r}{n+2}-1\right),\\
  n_1 &= \frac{tsr}{n}-4n_2,\\
  n_0 &= st-n_1-n_2 
\end{align*}
are non-negative integers satisfying $n_i/s\in \ZZ $ and $n_i/t \in \ZZ $ for $i=0,1,2$.
Moreover the quadratic polynomial,
\begin{align*}
  P(b)=&(s+t)^2\left(\frac{15}{n(n+2)(n+4)}+\frac{24b-3}{4n(n+2)}+\frac{2b^2-b}{2n}-\frac{b^2}{4}\right)\\
       &-2\left(n_1\left(\frac{1}{r}-\frac{1}{4}\right)\left(\frac{1}{r}+b\right)^2+n_2\left(\frac{4}{r}-\frac{1}{4}\right)\left(\frac{4}{r}+b\right)^2-n_0\frac{b^2}{4}\right)\\
       &-\frac{3}{4}(s+t)(1+b)^2 \le 0
\end{align*}
is non positive for all $b\in \RR$.
\end{lemma}

\begin{corollary}\label{20div3}
  There is no dual strongly perfect lattice $\Lambda\in\RR^{16}$ with $r(\Lambda)=20/3$.
\end{corollary}
\begin{proof}
	By Theorem~\ref{TableOfValues}  we have $s(\Lambda ) = s(\Lambda ^*) = 1296$. 
The polynomial $P(b)$ 
 from Lemma~\ref{lemma:polynomial_method} with $s=t=1296$ and $n=16$ is 
 $P(b)=-631800(b+7/325)(b+1/25)$ and satisfies $P(-8/325) > 0$, a contradiction.
\end{proof}

\begin{lemma}\label{32div5}
  There is no dual strongly perfect lattice $\Lambda\subset\RR^{16}$ with $r(\Lambda)=32/5$.
\end{lemma}
\begin{proof}
  By Remark~\ref{shortTableOfValues} there are $a,b\in\{1,2\}$ such that $s(\Lambda)=900\cdot a$,
  and $s(\Lambda^*)=900\cdot b$ such that $ab$ is even.
  These cases 
  yield a contradiction to Lemma~\ref{lemma:polynomial_method}.
\end{proof}

We now apply Lemma~\ref{IsLevelIntegral} to exclude the following cases.
\begin{lemma}\label{lemma:level}
  There is no dual strongly perfect lattice $\Lambda\subset\RR^{16}$ with 
    $$r(\Lambda) \in \left\{  48/7, 96/13, 15/2, 
                33/4, 192/23, 42/5,  60/7 \right\} $$
\end{lemma}
\begin{proof}
  Here we only present a proof for the case $r(\Lambda)=15/2$, as all the other cases can be excluded similarly.
   By Theorem~\ref{TableOfValues} there is some $a \in \{1,\ldots ,3\}$ such that $s(\Lambda)=512\cdot a$.
       We scale $\Lambda$ such that $\min(\Lambda)=1$.
           Let $\alpha\in\Lambda^*$, and write $(\alpha, \alpha)=\frac{p}{q}$ with coprime integers $p$ and $q$.
	       Then
	           \begin{align*}
			   (D4)(\alpha) &= {\frac{a\cdot 2^{4}\cdot {p}^{2}}{3\cdot {q}^{2}}} \in\ZZ
	       &&\Rightarrow q \mid 2^2, \\
			   \frac{1}{12}(D4-D2)(\alpha)&={\frac {a\cdot 2^{2}\cdot p \left( p-6\,q \right) }{  3^{2}\cdot {q}^{2}}} \in\ZZ
		       &&\Rightarrow q \mid 2, 3 \mid p.
		       \end{align*}
		           Let $\Gamma=\sqrt{\frac{2^2}{3}}\Lambda^*$. Then $\Gamma$ is an even lattice with $\min(\Gamma)=10, \min(\Gamma^*)=\frac{3}{2^2}$.
Similarly $\sqrt{ \frac{10\cdot 2^2}{3}}\Gamma^*$ is also even, which is impossible by Lemma~\ref{IsLevelIntegral}.
\end{proof}

Next, we can exclude the following cases.
\begin{lemma}\label{lemma:2adic_sublattice}
  There is no dual strongly perfect lattice $\Lambda\subset\RR^{16}$ with 
  $r(\Lambda)\in\{22/3, 54/7, 90/11, 26/3\}.$
\end{lemma}
\begin{proof}
  Here we give a proof for the case $r(\Lambda)=22/3$, as all other cases can be excluded similarly.
  Let $\Lambda$ be a dual strongly perfect lattice with $r(\Lambda)=22/3$.
  By Theorem~\ref{TableOfValues} we have  $s(\Lambda)=s(\Lambda ^*) = 1296$.
  We scale $\Lambda$ such that $\min(\Lambda)=2/3$, and put $\Gamma=\Lambda^*$.
  Then $\min(\Gamma)=11$, and for all $\alpha,\beta\in\Gamma$ holds
  \begin{align*}
	  (D4)(\alpha)&=3(\alpha,\alpha)^2\in\ZZ,\\
	  \frac{1}{6}(D13-D11)(\alpha,\beta) &= (\alpha,\beta)((\beta,\beta)-1) \in\ZZ.
  \end{align*}
So $(\alpha,\alpha)\in\ZZ$ for all $\alpha\in\Gamma$,
  and if $(\beta,\beta)$ is even, then $(\alpha,\beta)\in\ZZ$. 
  Let $\Gamma^{(e)}=\{\alpha\in\Gamma \mid (\alpha,\alpha)\in2\ZZ\}$.
  By $\frac{1}{6}(D13-D11)(\alpha,\beta)$ we see that $(\alpha ,\beta )\in \ZZ $ for all $\beta \in \Gamma ^{(e)}$, 
  $\alpha \in \Gamma $. In particular $\Gamma^{(e)}$ is an even sublattice of $\Gamma$ with $|\Gamma:\Gamma^{(e)}|=2^c, c \in\{1,2\}$
  (see for instance \cite[Lemma 2.8]{dim12}).
  So $\det(\Gamma)=2^{-2c}\det\Gamma^{(e)}$ and $\det\Gamma^{(e)}$ is an integer. Similarly $L=\sqrt{\frac{33}{2}}\Gamma^*$ has an
  even sublattice $L^{(e)}=\{\alpha\in L \mid (\alpha,\alpha)\in2\ZZ\}$ with $|L:L^{(e)}|=2^d, d \in\{1,2\}$.
  Therefore
  $$
  \det L^{(e)} = 2^{2d}\det L = \frac{2^{2(c+d)}\cdot33^{16}}{2^{16}\cdot\det\Gamma^{(e)}}\notin \ZZ,
  $$
  which is impossible.
\end{proof}

Next we employ the $k$-point semidefinite programming (SDP) bound for  spherical codes provided 
by de Laat et. al~\cite{deLaat_etal2018} to exclude the following case.

\begin{lemma}\label{144div7}
  There is no dual strongly perfect lattice $\Lambda\subset\RR^{16}$ with $r(\Lambda)=144/17$.
\end{lemma}
\begin{proof}
  By Remark~\ref{shortTableOfValues} 
  we get $s:=s(\Lambda)=s(\Lambda^*)=2312$ and put $r:=r(\Lambda) = 144/17$.
  Now fix some $\alpha\in\Min(\Lambda^*)$ and let $N_2(\alpha)=\{x\in\Min(\Lambda)\mid (x,\alpha)=2\}$.
  Then $$|N_2(\alpha)|=\frac{s r}{12\cdot 16}(\frac{r}{6}-1) =42.$$
  As in Lemma \ref{SphBndN_2} put
  $$
  \overline{N_2}(\alpha)
  =\left\{\sqrt{\frac{r}{r-4}}\left(x-\frac{2}{r}\alpha\right) \mid x\in N_2(\alpha)\right\}.
  $$
  Then $|\overline{N_2}(\alpha)|=|N_2(\alpha)|=42$, and 
  for any two distinct elements $\bar{x},\bar{y}\in\overline{N_2}(\alpha)$, we have $(\bar{x},\alpha)=0$,
  $(\bar{x},\bar{x})=1$, and
  \begin{align*}
  (\bar{x},\bar{y})
  =\frac{r}{r-4}\left(x-\frac{2}{r}\alpha, y-\frac{2}{r}\alpha\right)
  =\frac{r}{r-4}\left((x,y)-\frac{4}{r}\right) \leq 1/19 .
  \end{align*}
  Now using the $3$-point SDP bound for spherical codes ~\cite{deLaat_etal2018},
  we can compute that the cardinality of a spherical 
  $[-1,1/19]$-code in $S^{14}$ is upper bounded by $34$,
  which contradicts the fact that $|N_2(\alpha)|=42$. This concludes our proof.
\end{proof}

Now we use a different method to deal with the following case.

\begin{lemma}\label{case7}
  There is no dual strongly perfect lattice $\Lambda\subset \RR^{16}$ with $r(\Lambda)=7$.
\end{lemma}
\begin{proof}
  Let $\Lambda$ be a dual strongly perfect lattice in dimension $16$ with $r(\Lambda)=7$.
  By Theorem~\ref{TableOfValues} we have $s(\Lambda)=s(\Lambda^*)=1152$.
  We scale $\Lambda$ such that $\min(\Lambda)=1/2$ and $\min(\Lambda^*)=14$. Put $\Gamma:=\Lambda^*$. 
  Then for all $\alpha\in\Gamma$, 
  \begin{align*}
	  \frac{1}{12}(D4-D2)(\alpha) &= \frac{1}{4}(\alpha,\alpha)((\alpha,\alpha)-12) \in\ZZ.
  \end{align*}
  Thus $(\alpha,\alpha)$ is an even number, and $\Gamma$ is an even lattice;
  similarly $\sqrt{28}\Gamma^*$ is also even. 
  For any $\alpha\in\Min(\Gamma)$ and any $x\in\Min(\Lambda)$, define
  \begin{align*}
    N_2(\alpha)&:=\{x\in\Min(\Lambda)\mid (x,\alpha)=2\}, \text{ and}\\
    N_2(x)&:=\{\alpha\in\Min(\Gamma)\mid (x,\alpha)=2\}
  \end{align*}
  respectively. 
  Now fix $\alpha_1\in\Min(\Gamma)$ and assume that 
  \begin{align*}
  N_2(\alpha_1)&=\{x_1,x_2,x_3,x_4,x_5,x_6,x_7\},\\
  N_2(x_1)&=\{\alpha_1,\alpha_2,\alpha_3,\alpha_4,\alpha_5,\alpha_6,\alpha_7\}.
  \end{align*}
  By Corollary~\ref{linkomb}, we have $\sum_{j=1}^{7}\alpha_j = 28 x_1 $ and $\sum_{j=1}^{7}x_j = \alpha_1$.
  A simple calculation shows that $(x_i,x_j)=1/4$ for $1\leq i,j\leq 7$ and $i\neq j$, and
  $(\alpha_i,\alpha_j)=7$ for $1\leq i,j\leq 7$ and $i\neq j$.

  We claim that $|N_2(\alpha_1)\cap N_2(\alpha_2)|\le 1$. If not then there were two different vectors
  $x$ and $y$ in $N_2(\alpha_1)\cap N_2(\alpha_2)$. The Gram matrix formed by $x,y,\alpha_1,\alpha_2$ is
  $$
  \begin{pmatrix}
     1/2 & 1/4 & 2 & 2\\
     1/4 & 1/2 & 2 & 2\\
     2   &  2  & 14& 7\\
     2   &  2  & 7 &14
  \end{pmatrix},
  $$
  whose determinant is $-7/16$; but this is impossible as the Gram matrix should be positive-semidefinite.
  
  Since $|N_2(\alpha_1)\cap N_2(\alpha_j)|= 1$,
  we have $(x_i, \alpha_j)\in\{-2,-1,0,1\}$ for $2\le i,j\le 7$. 
  So $7=(x_i, 28x_1)= \sum_{j=1}^{7}(x_i, \alpha_j) = 2 + \sum_{j=2}^{7}(x_2, \alpha_j)\le 8$.
  Therefore, without loss of generality, we can assume that $(x_i,\alpha_i)=0$ for $2\le i\le 7$,
  and $(x_i,\alpha_j)=1$ for $2\le i,j \le 7$ and $ i\ne j$. 
  Because $(\alpha_2, x_1)=(\alpha_2, x_1-x_2)=2$, we can assume that
  $N_2(\alpha_2)=\{x_1,x_1-x_2,y_3,y_4,y_5,y_6,y_7\}.$
  Hence $(x_1, y_i)=1/4$ and $(x_2, y_i)=0$ where $3\le i\le 7$. 
  Similarly, assume that $N_2(x_2)=\{\alpha_1,\alpha_1-\alpha_2,\beta_3,\beta_4,\beta_5,\beta_6,\beta_7\}$.
  Hence $(\alpha_1,\beta_i)=7$ and $(\alpha_2, \beta_i)=0$ where $3\le i \le 7$.
  By the above argument used for $(x_i,\alpha_j)$, we can without loss of generality
  assume that $(x_i,\beta_i)=(y_i,\alpha_i)=0$ for $3\le i\le 7$, 
  and $(x_i,\beta_j)=(y_i,\alpha_j)=1$ for $3\le i,j\le 7$ with $i\ne j$.
  For $3\le i,j\le 7$ put $a_{ij}=(x_i,y_j), b_{ij}=(\alpha_i,\beta_j),$ and $c_{ij}=(y_i, \beta_j)$.
  Also we readily check that $a_{ij}\in\{a/28 \mid a \text{ is an integer and }-7 \le a\le 7 \}$,
  $b_{ij}\in\{-7,\dots,7\}$, and $c_{ij}\in\{-2,\dots,2\}.$
  Since every shortest vector $\alpha$ in $\Min(\Gamma)$ is equal to the sum of vectors
  in $N_2(\alpha)$, the lattice generated by vectors $x_1,\dots,x_6,\alpha_1,\dots,\alpha_6,y_3,\dots,y_6,\beta_3,\dots,\beta_6$
  is a sublattice of $\Gamma^*$; obviously it has minimum $1/2$.
  The Gram matrix formed by vectors $x_1,\dots,x_6,\alpha_1,\dots,\alpha_6,y_3,\dots,y_6,\beta_3,\dots,\beta_6$
  can be written as

  \begin{scriptsize}
  $$
    \kbordermatrix{
&x_1&x_2&x_3&x_4&x_5&x_6&\alpha_1&\alpha_2&\alpha_3&\alpha_4&\alpha_5&\alpha_6&y_3&y_4&y_5&y_6&\beta_3&\beta_4&\beta_5&\beta_6 \\
x_1&1/2&1/4&1/4&1/4&1/4&1/4&2&2&2&2&2&2&1/4&1/4&1/4&1/4&1&1&1&1\\
x_2&1/4&1/2&1/4&1/4&1/4&1/4&2&0&1&1&1&1&0&0&0&0&2&2&2&2\\
x_3&1/4&1/4&1/2&1/4&1/4&1/4&2&1&0&1&1&1&a_{33}&a_{34}&a_{35}&a_{36}&0&1&1&1\\
x_4&1/4&1/4&1/4&1/2&1/4&1/4&2&1&1&0&1&1&a_{43}&a_{44}&a_{45}&a_{46}&1&0&1&1\\
x_5&1/4&1/4&1/4&1/4&1/2&1/4&2&1&1&1&0&1&a_{53}&a_{54}&a_{55}&a_{56}&1&1&0&1\\
x_6&1/4&1/4&1/4&1/4&1/4&1/2&2&1&1&1&1&0&a_{63}&a_{64}&a_{65}&a_{66}&1&1&1&0\\
\alpha_1&2&2&2&2&2&2&14&7&7&7&7&7&1&1&1&1&7&7&7&7\\
\alpha_2&2&0&1&1&1&1&7&14&7&7&7&7&2&2&2&2&0&0&0&0\\
\alpha_3&2&1&0&1&1&1&7&7&14&7&7&7&0&1&1&1&b_{33}&b_{34}&b_{35}&b_{36}\\
\alpha_4&2&1&1&0&1&1&7&7&7&14&7&7&1&0&1&1&b_{43}&b_{44}&b_{45}&b_{46}\\
\alpha_5&2&1&1&1&0&1&7&7&7&7&14&7&1&1&0&1&b_{53}&b_{54}&b_{55}&b_{56}\\
\alpha_6&2&1&1&1&1&0&7&7&7&7&7&14&1&1&1&0&b_{63}&b_{64}&b_{65}&b_{66}\\
y_3&1/4&0&a_{33}&a_{43}&a_{53}&a_{63}&1&2&0&1&1&1&1/2&1/4&1/4&1/4&c_{33}&c_{34}&c_{35}&c_{36}\\
y_4&1/4&0&a_{34}&a_{44}&a_{54}&a_{64}&1&2&1&0&1&1&1/4&1/2&1/4&1/4&c_{43}&c_{44}&c_{45}&c_{46}\\
y_5&1/4&0&a_{35}&a_{45}&a_{55}&a_{65}&1&2&1&1&0&1&1/4&1/4&1/2&1/4&c_{53}&c_{54}&c_{55}&c_{56}\\
y_6&1/4&0&a_{36}&a_{46}&a_{56}&a_{66}&1&2&1&1&1&0&1/4&1/4&1/4&1/2&c_{63}&c_{64}&c_{65}&c_{66}\\
\beta_3&1&2&0&1&1&1&7&0&b_{33}&b_{43}&b_{53}&b_{63}&c_{33}&c_{43}&c_{53}&c_{63}&14&7&7&7\\
\beta_4&1&2&1&0&1&1&7&0&b_{34}&b_{44}&b_{54}&b_{64}&c_{34}&c_{44}&c_{54}&c_{64}&7&14&7&7\\
\beta_5&1&2&1&1&0&1&7&0&b_{35}&b_{45}&b_{55}&b_{65}&c_{35}&c_{45}&c_{55}&c_{65}&7&7&14&7\\
\beta_6&1&2&1&1&1&0&7&0&b_{36}&b_{46}&b_{56}&b_{66}&c_{36}&c_{46}&c_{56}&c_{66}&7&7&7&14
}.
$$
\end{scriptsize}

We attempt to complete this Gram matrix by adding the vectors $y_3,\dots,y_6,\beta_3,\dots,\beta_6$ each in turn.
For each vector, we should check that the Gram matrix of the completed vectors
is positive-semidefinite with rank $\le 16$, and the lattice with this Gram matrix has minimum $1/2$.
A brute-force search shows that there is no such Gram matrix.
This finishes our proof.
\end{proof}

Combining Theorem~\ref{TableOfValues}, Lemma~\ref{N_2N_2Dual},
 Corollary~\ref{20div3}, and Lemmas~\ref{32div5}-\ref{case7}, we obtain the following.
\begin{theorem}\label{rest}
  Let $\Lambda$ be a dual strongly perfect lattice in dimension $16$. Then 
  $$ r(\Lambda)\in\left\{6, \frac{36}{5},8,\frac{96}{11},9,\frac{64}{7}\right\}.$$
\end{theorem}

\section{Dual strongly perfect lattices of minimal type}\label{mintyp}
Let $\Lambda $ be some dual strongly perfect lattice of 
minimal type in dimension 16, so 
$$\Lambda \subset \RR ^{16},\ \min(\Lambda) \min(\Lambda^*) = 6 .$$
Put $m:=\min (\Lambda )$ and $d:= \min (\Lambda ^*) = 6/m $.
Let $s:= s (\Lambda ) $ and $t:= s(\Lambda^*)$.
The following arguments are only formulated to give
restrictions on $(s,t)$. The same conditions of course  also apply 
if we interchange $s$ and $t$. 

By the bounds on the kissing numbers we get 
 $8 \cdot 17 \leq s \leq 3678 $.
Moreover by equation $(D2)$ we have 
$smd/n = 3s/8 \in \ZZ $ so 

\begin{lemma}\label{lem0}
$8\mid s $. 
\end{lemma}

\begin{lemma}\label{lem1} 
Write  $s=2^aA$ with $A$ odd. 
If $A$ is squarefree then $a\geq 7$. 
\end{lemma}

\begin{proof}
Rescale $\Lambda $ such  that $m=3$ and $d=2$. 
Write $s=2^aA$ and assume that $A$ is odd, squarefree, and $a<7$. 
For $\alpha \in \Lambda ^*$ write $(\alpha , \alpha ) = \frac{p}{q} $ with
$\gcd (p,q ) =1 $. 
Then $\frac{1}{12}(D4-D2)$ implies that 
$$ \frac{s p }{2^7 q^2} (p-2q) = \frac{Ap}{2^{7-a}q^2} (p-2q)  \in \ZZ .$$
As $p$ and $q$ are coprime and $A$ is squarefree 
this implies that $q=1$ and $p$ is even.
So $\Lambda ^*$ is an even lattice with minimum $2$ 
so that its dual lattice $\Lambda $ has minimum 3. As $3>2$ 
and $\Lambda ^* \subseteq \Lambda $ this is a contradiction.
\end{proof}

\begin{lemma} \label{lem2}
	 Assume that $s=2^3b^2A$ with $A$ odd and squarefree, $b$ odd. 
	 The $2^9 $ divides $t$ and $b\geq 7$.
 \end{lemma}
 
 \begin{proof}
 Assume that $s=2^3b^2A$ with $A$ odd and squarefree, $b$ odd.
Rescale $\Lambda $ such  that $m=6/b$ and $d=b$. 
For $\alpha \in \Lambda ^*$ 
equation $\frac{1}{12}(D4-D2)$ implies that 
$$ \frac{A}{4} (\alpha ,\alpha ) ( (\alpha , \alpha )  -b )   \in \ZZ .$$
 As $A$ is odd and squarefree  
 this implies that  $(\alpha ,\alpha ) \in \ZZ $ and 
 $(\alpha ,\alpha ) \equiv 0 $ or $b $ $\pmod{4}$. 
 \\
 If $\alpha , \beta \in \Min (\Lambda ^*) $ then 
 $$(\alpha \pm \beta , \alpha \pm \beta ) =  
 2b \pm 2 (\alpha ,\beta ) $$ 
 are both either $b$ or $0$ modulo 4. 
 If $(\alpha ,\beta ) \in \frac{1}{2} + \ZZ $ then these are both odd and hence 
 $b\pmod{4} $ so their difference $4(\alpha ,\beta ) $ is $0\pmod{4}$ hence 
 $(\alpha ,\beta )\in \ZZ$ and $2b+2(\alpha ,\beta )$ is even, and hence $0\pmod{4}$
 implying that 
 \begin{center}
	 $(\alpha , \beta ) $ is odd for all $\alpha , \beta \in \Min (\Lambda ^* )$.
 \end{center} 
 As $\Lambda ^*$ is also strongly perfect and the fourth power of an odd integer is
 $1\pmod{16} $ we compute , 
 for any fixed $\alpha \in \Min (\Lambda ^*) $ 
 $$t \equiv _{16} \sum _{\beta \in \Min (\Lambda ^*) /{\pm 1 } } (\alpha ,\beta )^4  =  
 \frac{3tb^4}{16\cdot 18} \equiv _{16} \frac{t}{2^53} .$$ 
 So $32^5 t \equiv t \pmod{2^{5+4}} $ which implies that $2^9$ divides $t$. 
 \\
 Moreover if $b\leq 5$ then $(\alpha ,\beta ) = \pm 1$ for all $\alpha \neq \pm \beta \in \Min (\Lambda ^*)$ 
 and D2 gives us 
 $$\sum _{\beta \in \Min (\Lambda ^*) /{\pm 1 } } (\alpha ,\beta )^2 = 
 b^2 + (t-1) = \frac{b^2}{16} t $$ 
 which yields contradiction for $b=3,5$. 
 \end{proof}

 \begin{lemma}
	 If $3^2\not\mid s$ then $3^2 \mid t $. 
 \end{lemma} 

 \begin{proof}
 Assume that both $s$ and $t$ are not divisible by $3^2$. 
 Rescale $\Lambda $ such that $m=1$. For $\alpha \in \Lambda ^*$ 
 put $(\alpha,\alpha ) = \frac{p}{q}$, $\gcd(p,q) =1 $.
 Then $\frac{1}{12} (D4-D2) (\alpha ) $  yields that 
 $$\frac{s p}{2^73^2q^2} (p-6q ) \in \ZZ $$ 
 implying that $3 \mid p$. 
 So there is some $a\in \NN $ with
 $3\not\mid a  $ such that $\sqrt{\frac{a}{3}} \Lambda ^*$ is even. 
 Interchanging the role of $\Lambda $ and $\Lambda ^*$ we see that there
 is some $b\in \NN $ with $3\not\mid b$ such that $\sqrt{b} \Lambda $ is an 
 even lattice. Put $\Gamma :=\sqrt{b} \Lambda $. 
 Then $\Gamma $ is an even lattice such that 
 $$\sqrt{\frac{ab}{3}} \Gamma ^* = \sqrt{\frac{ab}{3}}  \frac{1}{\sqrt{b}} \Lambda ^*  
 = \sqrt{\frac{a}{3}} \Lambda ^* $$ is again even. 
 This is a contradiction as $ab/3$ is not an integer. 
 \end{proof} 

 Similarly we find

 \begin{lemma}
	 If $2^5\not\mid s$ then $2^5 \mid t $. 
 \end{lemma} 

 \begin{proof}
 Assume that both $s$ and $t$ are not divisible by $2^5$. 
 Rescale $\Lambda $ such that $m=1$. For $\alpha \in \Lambda ^*$ 
 put $(\alpha,\alpha ) = \frac{p}{q}$, $\gcd(p,q) =1 $.
 Then $\frac{1}{12} (D4-D2) (\alpha ) $  yields that 
 $$\frac{s p}{2^73^2q^2} (p-6q ) \in \ZZ $$ 
 implying that $ p$ is even. 
 So there is some odd $a\in \NN $ 
 such that $\Gamma:=\sqrt{{a}} \Lambda ^*$ is even. 
Moreover 
$\frac{1}{6}(D13-D11) (\alpha, \beta )$ shows that 
$$\frac{s}{2^63^2}(\alpha ,\beta ) ( (\alpha ,\alpha ) - 6) \in \ZZ $$ for all 
$\alpha,\beta \in \Lambda ^*$. 
In particular $$\Gamma ^{(e)}:=\{ \alpha \in \Gamma \mid (\alpha ,\alpha ) \in 4 \ZZ \} $$
is a sublattice of $\Gamma $ of index $1,2,$ or $4$ (see \cite[Lemma 2.8]{dim12}) 
and $\sqrt{1/2} \Gamma ^{(e) } $ is even. 
So $2^{12} $ divides the determinant of the even lattice $\Gamma = \sqrt{a} \Lambda ^*$. 
\\
Interchanging the role of $\Lambda $ and $\Lambda ^*$ 
we find that there is some odd $b\in \NN$ such that $\sqrt{6b} \Lambda $ is 
even and $2^{12} $ divides $\det (\sqrt{6b} \Lambda )$. 
All together 
$$2^{24} \mbox{ divides } \det (\sqrt{a} \Lambda ^* ) \det (\sqrt{6b} \Lambda )
= (6ab)^{16} $$ 
which contradicts the fact that $ab $ is odd. 
\end{proof}

\begin{lemma}
	If $s=2^aA$ with $A$ odd and squarefree and $a\leq 8$, 
	then $\Lambda ^*$ rescaled to minimum 4 is even and $36 $ divides $t$.
\end{lemma}

\begin{proof}
Rescale $\Lambda $ so that $m=3/2$ and $d=4$. 
Then for all $\alpha \in \Lambda ^*$ 
$$\frac{3s}{2^7} (\alpha,\alpha )^2 \in \ZZ \mbox{ and } 
\frac{s}{2^9} ((\alpha ,\alpha ) ((\alpha ,\alpha ) - 4) ) \in \ZZ $$
so $(\alpha ,\alpha ) \in 2\ZZ $. 
Moreover for any $\alpha \in \Min (\Lambda ^*) $ 
the set $N_2(\alpha ):= \{ \beta \in \Min (\Lambda ^*) \mid (\alpha , \beta ) = 2 \} $ 
has cardinality 
$$\frac{5t}{36} - 20  $$ 
which implies that $36 \mid t $.
\end{proof}

\begin{lemma}
	$s\neq 648=2^3\cdot 3^4$.
\end{lemma}

\begin{proof}
	Assume that $s=2^33^4$ and rescale $\Lambda $ so that $m=2/3$ and $d=9$. 
	Then for all $\alpha \in \Lambda ^*$ we get 
	$$\frac{1}{12}(D4-D2)(\alpha ) = \frac{1}{4} (\alpha ,\alpha ) ((\alpha,\alpha ) - 9 )  
	\in \ZZ .$$
	In particular
	for all $\alpha , \beta \in \Min (\Gamma )$  
	$$(\alpha \pm \beta , \alpha \pm \beta ) = 18 \pm 2 (\alpha , \beta )  \equiv 0 \mbox{ or } 1 \pmod{4} $$
	which implies that  $(\alpha ,\beta ) $ is an odd integer for all 
	$\alpha ,\beta \in \Min (\Gamma )$. As $|(\alpha ,\beta )| \leq  \frac{9}{2}  = 4.5 $ 
	we find that 
	$$(\alpha , \beta ) \in \{ \pm 3  , \pm 1 \} .$$
	As also $\Min(\Gamma ) $ is a 4-design, for any fixed $\alpha \in \Min (\Gamma ) $
	the integers $t=|\Min(\Gamma) |/2 $, $n_i:= |\{ \beta \in \Min(\Gamma) \mid 
	(\alpha ,\beta ) = i \} |$ satisfy 
	$$\begin{array}{ccccccc} 
		1 &  + &  n_1 &  + &  n_3 &   = & t \\ 
		9^2 & + & n_1 & + & 3^2 n_3 & = & \frac{9^2}{16} t \\
		9^4 & + & n_1 & + & 3^4 n_3 & = & \frac{3\cdot 9^4}{16\cdot 18} t 
	\end{array} $$
	This equation has a unique solution 
	$(n_1,n_3,t) = \frac{1}{19} (2187 , 1890 , 4096 )  $ 
	which is of course absurd.
\end{proof}

Now an application of the above lemmas leads to the following list of 118 possible pair $(s,t)$
of a dual strongly perfect lattice $\Lambda\subset\RR^{16}$ of minimal type. (WLOG we assume
that $s\leq t$.)
\begin{enumerate}[(1)]
  \item $s=144,t=128\cdot i, 2\leq i \leq 26$.
  \item $s=144, t=288\cdot i, 1\leq i\leq 11.$
  \item $s=144,t=800\cdot i,1\leq i\leq 3.$
  \item $s=144,t=1568\cdot i,1\leq i\leq 2$.
  \item $s=256,t=144\cdot i, 2\leq i\leq 16$.
  \item $s=288,t=128\cdot i, 3\leq i\leq 16$.
  \item $s=288,t=144\cdot i, 2\leq i\leq 14$.
  \item $s=288,t=400\cdot i, 1\leq i\leq 5$.
  \item $s=288,t=784\cdot i, 1\leq i\leq 2$.
  \item $s=288,t=1936$.
  \item $s=384,t=144\cdot i, 3\leq i\leq 10$.
  \item $s=400,t=288\cdot i, 2\leq i\leq 4$.
  \item $s=432,t\in\{512,576,640,768,800,864,896,1024,1152\}$.
  \item $s=512,t\in\{576,720,864\}$.
  \item $s=576,t\in\{576,640,720,768,784,800\}$.
  \item $s=640,t=720$.
\end{enumerate}
Among those 118 possible pairs of values,
\begin{enumerate}
  \item there are 54 possible pairs of values with the property that 
	  either $\Lambda$ or $\Lambda^*$ rescaled to minimum 4 is even and the even level of $\Lambda $ (or $\Lambda ^*$) 
	  divides 24. So one of $\Lambda $ or $\Lambda ^*$
 is an even lattice whose dual has minimum $\geq 3/2$. 
    Then by Lemma~\ref{minge3/2} we know that $\Lambda$ or $\Lambda ^*$ 
    are similar to one of $\Gamma _{16}$ and $O_{16}$.
  \item there are 23 possible pairs of values with the following property: if we rescale $\Lambda$
    with minimum $6$ then $\Lambda$ is even and the even level of $\Lambda$ divides $6$. 
    Then by Lemma~\ref{leveldiv6} we know that there is no such $\Lambda$.
  \item for the remaining 41 cases, a direct application of the modular form approach described in the next section shows that
	  there is no such pair $(\Lambda , \Lambda ^*)$.
\end{enumerate}

In summary we have proved the following.
\begin{theorem} \label{r6}
  Let $\Lambda$ be a dual strongly perfect lattice in dimension $16$ and of minimal type.
  Then $\Lambda$ is isomorphic to one of $\Gamma_{16}$, $\Gamma_{16}^{*}$, $O_{16}$ or $O_{16}^*$.
\end{theorem}

\section{Modular forms and $\vartheta$-series}\label{ModForm}
Let $\Lambda \leq \RR^n$ be an even lattice. Throughout this section we will assume that $n=2k, k\in \ZZ_{>0}$ for simplicity.
We associate to $\Lambda$ a holomorphic function  on the upper half plane
$\mathbb{H} = \{\tau \in \CC \mid \textup{Im}\ \tau > 0\} \subset \CC.$
For $\tau \in \mathbb{H}$ let $q = e^{2\pi i\tau}.$
The \emph{theta series} of $\Lambda$ is the function
$$
\vartheta_{\Lambda}(\tau)=\sum_{x\in\Lambda}q^{\frac{1}{2}(x,x)}\quad \text{ for } \tau\in\mathbb{H}.
$$

A nice introduction to the relevant theory is the book \cite{Ebeling}, from 
which we also borrow the notation.  In particular we need the following theta transformation
formula relating the theta series of a lattice and its dual lattice. 
\begin{lemma}\label{TransFormula}\cite[Proposition 2.1]{Ebeling}
$$
\vartheta_{\Lambda}\left(-\frac{1}{\tau} \right) = \left(\frac{\tau}{i}\right)^{k}\,\sqrt{\det \Lambda^*}\,\vartheta_{\Lambda^*}(\tau).
$$
\end{lemma}

\begin{theorem}\label{EvenLatticeThaSes}(\cite[Theorem 3.2]{Ebeling})
Let $\Lambda$ be an even lattice of even level $\ell $. Then 
the theta series of $\Lambda $ is in the space of modular forms of weight $k$ for 
the subgroup $\Gamma _0(\ell )$ to some character $\chi _{\Lambda }$ only depending 
on $\det(\Lambda )$
$$\vartheta_{\Lambda}(\tau)\in \mathcal{M}_{k}(\Gamma_0(\ell ),\chi_{\Lambda}),\text{ where }
\chi_{\Lambda}(\cdot) = \left( \frac{(-1)^{k}\det(\Lambda)}{\cdot}\right).
$$
\end{theorem}
The matrix
$$
W_\ell  = \left(\begin{array}{ll}
        0 & 1/\sqrt{\ell }\\
        -\sqrt{\ell } & 0
        \end{array}\right) \in \text{SL}_2(\RR)
$$
is called the $\ell $-th \emph{Atkin-Lehner operator}. 
The well-known action of the Atkin-Lehner operator on 
the theta series of an even lattice $\Lambda$ of even level $\ell $ and dimension $n=2k$
is directly obtained from Lemma \ref{TransFormula}.
\begin{proposition}
\label{AtkinLehner}
$$
\vartheta_{\Lambda}(\tau)\mid_kW_\ell  = \left(-\frac{\sqrt{\ell }}{i}\right)^k\sqrt{\det(\Lambda^*)}\vartheta_{\sqrt{\ell }\Lambda^*}(\tau).
$$
\end{proposition}

\begin{theorem}\label{DiffOfThaSes}\cite[p.376]{Si}\cite{Walling}
Let $\Lambda$ be an even lattice of even level $\ell $ and dimension $2k$.
If lattices $\Lambda$ and $\Lambda'$ are in the same genus, then 
$$\vartheta_{\Lambda}(\tau)-\vartheta_{\Lambda'}(\tau)\in\mathcal{S}_{k}(\Gamma_0(\ell ),\chi_{\Lambda})$$
where, as usual, ${\mathcal S}_k$ denotes the cuspidal subspace of the space of modular forms ${\mathcal M}_k$.
\end{theorem}

Now we describe how to employ the theory of modular forms to exclude the existence of
a dual strongly perfect lattice. Let $\Lambda $ be a dual strongly perfect lattice.
Let $s = s(\Lambda) = \frac{1}{2} |\Min (\Lambda) |$ be half of the kissing number of $\Lambda$, 
let $s'=s(\Lambda^*)= \frac{1}{2} |\Min (\Lambda^*) |$ 
and $r(\Lambda ) = \min (\Lambda) \min (\Lambda^*) = r(\Lambda ^*)$ 
be the Berg\'e-Martinet invariant of $\Lambda$. 
We write $r(\Lambda ) =m \cdot d$ such that when rescaled to minimum $\min(\Lambda^*)=m$ 
the lattice $\Lambda^*$ is  even and in particular contained in its dual lattice $\Lambda$ 
(which is then of minimum $d$). 
We then interchange the roles of $\Lambda $ and $\Lambda ^*$ to obtain 
a factorization $r(\Lambda ) = m' \cdot d' $ 
such that $\Lambda$ is even if rescaled to $\min (\Lambda) = m' $. 
In the latter scaling the even level of $\Lambda $ divides $m/d'$ and in particular
$$\text{exp}(\Lambda^*/\Lambda) \text{ divides } \frac{m}{d'}.$$
We also obtain a finite list of possible determinants of $\Lambda$ from the 
upper bound on the Hermite constant $\gamma _n$, 
more precisely a finite list of possible invariants of the finite abelian group $\Lambda^*/\Lambda$.
For each invariant it is easy to read off all possible genera of lattices, 
given by the $p$-adic genus symbols for all primes $p$ dividing 
$2 \det(\Lambda )$ (see \cite[Chapter 15]{SPLAG}).
As each genus only contains finitely many isometry classes of lattices, 
one might in principle enumerate all of them. But usually there are far too many classes.

Here the theory of modular forms comes into play.
Rescale $\Lambda$ with $\min(\Lambda)=m'$ such that $\Lambda$ is even and denote $\ell =m/d'$. 
Then we know that
\begin{align}\label{equ:thetaseriees}
  \begin{split}
    \vartheta_{\Lambda}(\tau) &= 1 + 2sq^{m'}+ \dots,\\
    \vartheta_{\sqrt{\ell }\Lambda^*}(\tau) &= 1 + 2s'q^{m}+ \dots.
  \end{split}
\end{align}
By Theorem~\ref{EvenLatticeThaSes} both $\vartheta_{\Lambda}(\tau)$ and $\vartheta_{\sqrt{\ell }\Lambda^*}(\tau)$ lie in
the finite dimensional vector space $\mathcal{M}_{k}(\Gamma_0(\ell ),\chi_{\Lambda})$, 
of which one can explicitly compute a basis (for instance with {\sc Magma} \cite{Magma}). 
One can decompose $\vartheta_{\Lambda}(\tau)$ as 
$$\vartheta_{\Lambda}(\tau)=E(\tau)+C(\tau)$$ 
where $E(\tau)=\sum_{j=0}^{\infty}a_{E}(j)q^j$ is an Eisenstein series, 
and $C(\tau)=\sum_{j=1}^{\infty}a_{C}(j)q^{j}$ is a cusp form.

Now for each genus, we can either find a representative lattice in this genus or 
compute the genus theta series, i.e., the weighted average over all theta series in the genus. 
The genus theta series is an Eisenstein series, and its Fourier coefficients $a_E(j)$ can be 
computed as a product
\begin{align*}
  a_E(j) = \prod_{p\le\infty}\beta_p(j)
\end{align*}
of local densities $\beta_p(j)$. 
We use the formulas of Yang~\cite{Yang1998} to compute these local densities and 
then use the Sage computeralgebrasystem \cite{Sage} to compute the Fourier coefficients $a_{E}(j)$. 

Assume that the cusp forms subspace $\mathcal{S}_{k}(\Gamma_0(\ell ),\chi_{\Lambda})$ is of dimension $h$
and it has a basis $\{B_i(\tau)\}_{i=1}^{h}$, where $B_{i}(\tau)=\sum_{j=0}^{\infty}a_{B_i}(j)q^j$. 
As $C(\tau)\in \mathcal{S}_{k}(\Gamma_0(\ell ),\chi_{\Lambda})$,
we can write that
$$ C(\tau) = \sum_{i=1}^{h} c_i B_{i}(\tau)=\sum_{j=0}^{\infty}\sum_{i=1}^{h}c_ia_{B_i}(j)q^j$$
as a linear combination of the basis $\{B_i(\tau)\}_{i=1}^{h}$. Hence
\begin{align*}
\vartheta_{\Lambda}(\tau) &= E(\tau)+C(\tau)=\sum_{j=0}^{\infty}(a_E(j)+\sum_{i=1}^{h}c_ia_{B_i}(j))q^j.
\end{align*}
We write $E(\tau)\mid_{k}W_{l}=\sum_{j=0}^{\infty} a_{E^W}(j)q^j$ and 
$B_{i}(\tau)\mid_{k}W_{l}=\sum_{j=0}^{\infty}a_{B_i^W}(j)q^j$. Then
\begin{align*}
\vartheta_{\Lambda}(\tau)\mid_{k} W_{l} 
= E(\tau)\mid_{k}W_{l}+\sum_{i=1}^{h}c_iB_{i}(\tau)\mid_{k}W_{l}
=\sum_{j=0}^{\infty}(a_{E^W}(j)+\sum_{i=1}^{h}c_ia_{B_i^W}(j))q^j.
\end{align*}
Note that these coefficients $a_{E^W}(j)$
and $a_{B_i^W}(j)$ can be very easily computed from those coefficients $a_{E}(j)$ and $a_{B_i}(j)$.

Set $\text{const}=\left(-\frac{\sqrt{\ell }}{i}\right)^k\sqrt{\det(\Lambda^*)}$. 
Now by Proposition~\ref{AtkinLehner} and the above discussion,
we get the following linear restrictions on those variables $c_{i}$:
\begin{align}\label{LP}
  \begin{cases}
  a_E(0)+\sum_{i=1}^{h}c_ia_{B_i}(0)&=1,\\
  a_E(j)+\sum_{i=1}^{h}c_ia_{B_i}(j)&=0, \text{   for } 1\leq j\leq m'-1,\\ 
  a_E(m')+\sum_{i=1}^{h}c_ia_{B_i}(m')&=2s,\\ 
  a_E(j)+\sum_{i=1}^{h}c_ia_{B_i}(j)&\geq0, \text{    for } j\geq m'+1,\\ 
  a_{E^W}(0)+\sum_{i=1}^{h}c_ia_{B_i^W}(0)&=1\cdot \text{const},\\
  a_{E^W}(j)+\sum_{i=1}^{h}c_ia_{B_i^W}(j)&=0,\text{    for } 1\leq j\leq m-1\\
  a_{E^W}(m)+\sum_{i=1}^{h}c_ia_{B_i^W}(m)&=2s'\cdot\text{const},\\
  a_{E^W}(j)+\sum_{i=1}^{h}c_ia_{B_i^W}(j)&\geq0, \text{    for } j\geq m+1.
  \end{cases}
\end{align}
Now we employ the lrs Version 7.0~\cite{lrs} to check whether there is any feasible solution for
thoses variables $c_{i}$. (In practice we will only use the coefficients up to degree 100.)
If there is no feasible solution, then we conclude that there is no such lattice $\Lambda$ with
the corresponding genus symbol.

To illustrate the modular forms technique 
we will prove that there is no dual strongly perfect lattice $\Lambda\subset\RR^{16}$
with $r(\Lambda)\in\{96/11,64/7\}$ in the following.

\begin{lemma}\label{96div11}
  There is no dual strongly perfect lattice $\Lambda$ with $r(\Lambda)=96/11$.
\end{lemma}
\begin{proof}
  By Remark~\ref{shortTableOfValues} there are $a,b\in\{3,4,6,8,9\}$ with $24\mid ab$  such that
  $s(\Lambda)=242\cdot a$ and $s(\Lambda^*)=242\cdot b$.
  We scale $\Lambda$ such that $\min(\Lambda)=1$.
  Let $\alpha\in\Lambda^*$, and write $(\alpha, \alpha)=\frac{p}{q}$ with coprime integers $p$ and $q$.
  Then
  \begin{align*}
    (D4)(\alpha) &= {\frac{a\cdot 11^{2}\cdot {p}^{2}}{2^{4}\cdot 3\cdot {q}^{2}}} \in\ZZ, \\
    \frac{1}{12}(D4-D2)(\alpha) &= {\frac{a\cdot 11^{2}\cdot p \left( p-6\,q \right) }{2^{6}\cdot 3^{2}\cdot {q}^{2}}} \in\ZZ.
  \end{align*}
  Hence we have:
  \begin{enumerate}[(i)]
      \item If $a = 9$, then $2^5 \mid p, q \mid 11$, whence $\sqrt{\frac{11}{2^4}}\Lambda^*$ is even with minimum $6$.
      \item If $a\in\{4,8\}$, then $6\mid p, q\mid 11$, whence $\sqrt{\frac{11}{3}}\Lambda^*$ is even with minimum $32$.
      \item If $a\notin\{4,8,9\}$, then $2^{4}3\mid p, q \mid 11$, whence 
            $\sqrt{\frac{11}{2^3\cdot3}}\Lambda^*$ is even with minimum $4$.
  \end{enumerate}
  We first treat the case where $a \ne 9$ and $b\ne 9$. 
  Then $\Gamma=\sqrt{\frac{11}{3}}\Lambda^*$ is even with $\min(\Gamma)=32$.
  Similarly, $\sqrt{32\cdot\frac{11}{3}}\Gamma^*$ is also even, which is impossible by
  Lemma~\ref{IsLevelIntegral}. This leaves us only two cases $a=8,b=9$ or $a=9,b=8$.
  By symmetry we assume that $a=8$ and $b=9$. 
  Then $\Gamma=\sqrt{\frac{11}{3}}\Lambda^*$ is even with $\min(\Gamma)=32$.
  Similarly, $\sqrt{22}\Gamma^*$ is also even with minimum $6$.

  Denote $L=\sqrt{22}\Gamma^*=\sqrt{6}\Lambda$. Then $\det L=2^a11^b,$ where $a,b\in\{0\dots16\}.$
  By Lemma~\ref{DeterminantBound}, we get 
  $\det L\in\{2^2 11^4,2^3 11^4,2^6 11^3,2^9 11^2,2^{13} 11,2^{16}\}$.
  If $\det L= 2^{16}$ then $\frac{1}{\sqrt{2}} L$ is a unimodular lattice, and the minimum of it
  cannot exceed $2$, therefore $\min(L)\le 4$, which contradicts the fact that $\min(L)=6$.
  Now by reading off all possible genera of $L$ with 
  $\det L\in\{2^2 11^4,2^3 11^4,2^6 11^3,2^9 11^2,2^{13} 11\}$,
  we find only two possible genera $g_1=$\,\RNum{2}$_{16,0}(2^{+2}11^{+4})$ and $g_2=$\,\RNum{2}$_{16,0}(2^{-2}11^{-4})$.
  We calculate the genus theta series of $g_i,1\leq i\leq 2$ and get
  \begin{align*}
    E_{g_1}(\tau) &= 
    1 + \frac{14999208}{7591877}q + \frac{1950015144}{7591877}q^2 + \frac{32818267104}{7591877}q^3 +
    \frac{249632054952}{7591877}q^4 + O(q^5),\\\\
    E_{g_2}(\tau) &= 1 + \frac{1248806}{622285}q + \frac{157378598}{622285}q^2 + 
                     \frac{2732387528}{622285}q^3 + \frac{20141991974}{622285}q^4 + O(q^5).
  \end{align*}
  Then 
  $C_i(\tau)=\vartheta_{L}(\tau)-E_{g_i}(\tau)\in\mathcal{S}_{8}(\Gamma_0(22),\chi)$ if the genus symbol
  of $L$ is $g_i$, where $\chi$ is the trivial character.
  The subspace $\mathcal{S}_{8}(\Gamma_0(22),\chi)$ is of 19 dimension.
  We also know that
  \begin{align*}
    \vartheta_{L} &= 1+2\cdot 242\cdot8 q^{3}+O(q^4),\\
    \vartheta_{\sqrt{22}L^*} &= 1+2\cdot242\cdot 9\,q^{16}+O(q^{17}).
  \end{align*}
  Now we use lrs to solve the linear restrictions~\eqref{LP}, and find that there does not exist
  cusp forms $C_{i}(\tau)$ which satisfies those restrictions~\eqref{LP}. This concludes our proof.
\end{proof}

\begin{lemma}\label{64div7}
  There is no dual strongly perfect lattice $\Lambda$ with $r(\Lambda)=64/7$.
\end{lemma}
\begin{proof}
  By Remark~\ref{shortTableOfValues} there is some $a\in\{1,\ldots ,4\}$ such that $s(\Lambda)=882\cdot a$.
  We scale $\Lambda$ such that $\min(\Lambda)=1/7$ and $\min(\Lambda^*)=64.$
  For every $\alpha\in\Gamma=\Lambda^*$, 
  \begin{align*}
	  (D4)(\alpha)=\frac{3a}{2^4}(\alpha,\alpha)^2\in\ZZ\implies (\alpha,\alpha)\in2\ZZ.
  \end{align*}
  Hence $\Gamma$ is even. Now rescale $\Lambda$ such that $\min(\Lambda)=1$ and $\min(\Lambda^*)=64/7.$
  Then $\sqrt{64}\Lambda$ is even, and hence for $x,y\in\Min(\Lambda)$ with $x\ne \pm y$,
  $$(x,y)\in\{a/64: a\in\ZZ, -32\leq a\leq 32\}.$$
  Now fix $\alpha\in\Min(\Lambda^*)$, and let $N_i(\alpha)=\{x\in\Min(\Lambda)\mid (x,\alpha)=i\}$
  for $i\in\{2,3\}$. Note that 
  \begin{align}\label{64div7_eq1}
    6|N_3(\alpha)|+|N_2(\alpha)|=22\cdot a.
  \end{align}

  We first prove that $N_3(\alpha)=\emptyset$.
  We claim that if $N_3(\alpha)\neq\emptyset$ then $|N_3(\alpha)|=1$ and $|N_2(\alpha)|=0$. 
  Assume that there exist 
  two distinct vectors $x,y\in N_3(\alpha)$. Write the Gram matrix formed by those three vectors
  $x,y,\alpha$ as
  \begin{align*}
    G=\begin{pmatrix}
       1/7 & (x,y) & 3 \\
       (x,y) & 1/7 & 3 \\
       3   &  3  & 64
    \end{pmatrix}.
  \end{align*}
  A computer calculation shows that $G$ is positive semidefinite only when $(x,y)=9/64$,
  but then the lattice $L:=\langle x,y,\alpha\rangle \subset \Lambda$ has minimum 
  $1/224<\min(\Lambda)=1/7$, which is impossible. A similar computation shows that
  $N_2(\alpha)=\emptyset$. 
  Therefore $6|N_3(\alpha)|+|N_2(\alpha)|=6$, but this contradicts~\eqref{64div7_eq1}.

  As $N_3(\alpha)=\emptyset$, by Lemma~\ref{SphBndN_2} we have $|N_2(\alpha)|\le 61.9$,
  so $a \le 2$. Similarly $b\le 2$ and by Remark~\ref{shortTableOfValues} we have $a=b=2$.

  Recall that we scaled $\Lambda $ such that $\min (\Lambda ) = 1$. 
For $\alpha \in \Lambda ^*$  write $(\alpha,\alpha ) = \frac{p}{q} \in \QQ $. 
Then equation $(D4)(\alpha ) $ and $\frac{1}{12}(D4-D2)(\alpha )$ yield 
$$\frac{3\cdot 7^2}{2^3} \frac{p^2}{q^2} \in \ZZ \mbox{ and } \frac{7^2}{2^5} \frac{p}{q} (\frac{p}{q}-6) \in \ZZ $$ 
whence 
     $2^4 \mid p$ and $ q\mid 7$. 
     In particular $\Gamma := \sqrt{\frac{7}{2^3}}\Lambda^*$ is an even
          lattice with minimum $2^3$.
          Similarly, the lattice $\sqrt{7}\Gamma^*$ is also even.
          By Lemma~\ref{DeterminantBound}, we have $\det \Gamma = 7^8$. As the even level of $\Gamma $ is
	  $7$ the lattice $\Gamma $ is in the genus of even $7$-modular lattices represented by 
	  $L = E_8 \perp \sqrt{7} E_8$.
          The $\vartheta$-series of $L$ is
          \begin{align*}
            \vartheta_{L} &= 1 + 240q + 2160q^2 + 6720q^3 + 17520q^4 + 30240q^5 + 60480q^6 + 82800q^7+O(q^8).
          \end{align*}
          Because of Theorem~\ref{EvenLatticeThaSes} we see that $\vartheta_{L}\in\mathcal{M}_{8}(\Gamma_0(7),\chi)$ 
          where $\chi$ is trivial. With Theorem~\ref{DiffOfThaSes} it follows that 
          $S=\vartheta_{L}-\vartheta_{\Gamma}\in\mathcal{S}_{8}(\Gamma_0(7),\chi)$.
          The subspace $\mathcal{S}_{8}(\Gamma_0(7),\chi)$ is of dimension $3$.
          We know that
          \begin{align*}
            \vartheta_{\Gamma} &= 1+2\cdot 882\cdot 2 q^{4}+O(q^5),\\
            \vartheta_{\sqrt{7}\Gamma^*} &= 1+2\cdot 882 \cdot 2 q^{4}+O(q^{5}),\\
            S\mid_8 W_{7} &= \vartheta_{L}\mid_8 W_{7} - \vartheta_{\Gamma}\mid_8 W_{7}
                           = (\vartheta_{\sqrt{7}L^*}-\vartheta_{\sqrt{7}\Gamma^*}).
          \end{align*}
          Then we get $8$ relations on the coefficients of $S$. The MAGMA computation shows that there
          is no solution for these $8$ relations.
\end{proof}

\section{$r(\Lambda)=\frac{36}{5}$}\label{36div5}
Let $\Lambda$ be some dual strongly perfect lattice of dimension 16 with $r(\Lambda)=36/5$.
By Theorem~\ref{TableOfValues} $s(\Lambda),s(\Lambda^*)\in\{400,800,1200\}.$
WLOG we assume that $s(\Lambda)\leq s(\Lambda^*)$.
We first apply the modular form approach and obtain the following.

\begin{lemma}
  There is no dual strongly perfect lattice $\Lambda\subset\RR^{16}$ with $r(\Lambda)=36/5$ and
  \begin{enumerate}
   \item $s(\Lambda)=400$, $s(\Lambda^*)\in\{400,800\}$.
   \item $s(\Lambda)=800$, $s(\Lambda^*)=1200$.
  \end{enumerate}
\end{lemma}

Next we apply the technique from Lemma~\ref{case7} to exclude two more pairs of values. 

\begin{lemma}
  There is no dual strongly perfect lattice $\Lambda\subset\RR^{16}$ with $r(\Lambda)=36/5$
  and $(s(\Lambda),s(\Lambda^*))\in\{(400,1200),(800,800)\}$.
\end{lemma}
\begin{proof}
  Here we only give the proof for the case that $s(\Lambda)=s(\Lambda^*)=800$, as the other case
  can be proved similarly.
  We scale $\Lambda$ such that $\min(\Lambda)=3/5, \min(\Lambda^*)=12$. Then $\Lambda^*$ is an even
  lattice; similarly $\sqrt{20}\Lambda$ is even. Thus for $x,y\in\Lambda$ we have $(x,y)\in\frac{1}{20}\ZZ$.
  Choose $\alpha_1\in\Lambda$ and put $N_2(\alpha_1)=\{x_1,x_2,x_3,x_4,x_5,x_6\}$.
  We know that $\sum_{i=1}^{6}x_i=\alpha_1$, so $3/5+\sum_{i=2}^{6}(x_1,x_i)=2$.
  Combining this with $(x_1,x_i)\in\frac{1}{20}\ZZ$, we readily check
  there are only two possibilities for the multiset $\{(x_1,x_i): i\in\{2\dots6\}\}$:
  $\{1/5, (3/10)^4\}$ and $\{(1/4)^2, (3/10)^3\}$, where the exponents indicate multiplicities.
  Using this observation, we easily find that there are totally four possible Gram matrix
  formed by vectors $x_1,\dots,x_6$ up to the permutation equivalence:
  \begin{scriptsize}
  $$
  \kbordermatrix{
       &x_1&x_2&x_3&x_4&x_5&x_6\\
    x_1&3/5&1/5&3/10&3/10&3/10&3/10\\
    x_2&1/5&3/5&3/10&3/10&3/10&3/10\\
    x_3&3/10&3/10&3/5&1/5&3/10&3/10\\
    x_4&3/10&3/10&1/5&3/5&3/10&3/10\\
    x_5&3/10&3/10&3/10&3/10&3/5&1/5\\
    x_6&3/10&3/10&3/10&3/10&1/5&3/5
  }, \quad
  \kbordermatrix{
       &x_1&x_2&x_3&x_4&x_5&x_6\\
    x_1&3/5&1/5&3/10&3/10&3/10&3/10\\
    x_2&1/5&3/5&3/10&3/10&3/10&3/10\\
    x_3&3/10&3/10&3/5&1/4&1/4&3/10\\
    x_4&3/10&3/10&1/4&3/5&3/10&1/4\\
    x_5&3/10&3/10&1/4&3/10&3/5&1/4\\
    x_6&3/10&3/10&3/10&1/4&1/4&3/5
  },
  $$
  $$
  \kbordermatrix{
       &x_1&x_2&x_3&x_4&x_5&x_6\\
    x_1&3/5&1/4&1/4&3/10&3/10&3/10\\
    x_2&1/4&3/5&1/4&3/10&3/10&3/10\\
    x_3&1/4&1/4&3/5&3/10&3/10&3/10\\
    x_4&3/10&3/10&3/10&3/5&1/4&1/4\\
    x_5&3/10&3/10&3/10&1/4&3/5&1/4\\
    x_6&3/10&3/10&3/10&1/4&1/4&3/5
  }, \quad
  \kbordermatrix{
       &x_1&x_2&x_3&x_4&x_5&x_6\\
    x_1&3/5&1/4&1/4&3/10&3/10&3/10\\
    x_2&1/4&3/5&3/10&1/4&3/10&3/10\\
    x_3&1/4&3/10&3/5&3/10&1/4&3/10\\
    x_4&3/10&1/4&3/10&3/5&3/10&1/4\\
    x_5&3/10&3/10&1/4&3/10&3/5&1/4\\
    x_6&3/10&3/10&3/10&1/4&1/4&3/5
  }.
  $$
  \end{scriptsize}

  Put $N_2(x_1)=\{\alpha_1,\dots,\alpha_6\}.$ We also have four possible Gram matrix
  $20A_1,\dots,20A_4$ up to permutation equivalence. Considering the Gram matrix formed by
  vectors $x_1,\dots,x_6,\alpha_1,\dots,\alpha_6$, we find totally $20$ possible such matrix
  up to relabelling of vectors $x_2,\dots,x_6$ and $\alpha_2,\dots,\alpha_6$, by checking whether
  it is positive-semidefinite and the lattice with this Gram matrix has minimum norm not less
  than $3/5$. Put $N_2(\alpha_2)=\{x_1,y_2,\dots,y_6\}.$ We continue to investigate the Gram matrix formed by vectors 
  $x_1,\dots,x_6,\alpha_1,\dots,\alpha_6,x_1,y_2,\dots,y_6$. Direct computation shows none of
  these 20 matrices can be completed to such a Gram matrix. This finishes our proof.
\end{proof}

The only remaining situation is $s(\Lambda ) = s(\Lambda ^*) =1200$. 
Here the proof of \cite[Theorem 8.1]{BachocVenkov} applies almost literally to obtain: 

\begin{lemma}
  If $\Lambda$ is a dual strongly perfect lattice of dimension $16$ with $r(\Lambda)=36/5$ and
  $s(\Lambda)=s(\Lambda^*)=1200$, then $\Lambda\cong N_{16}$.
\end{lemma}

\section{The case $r(\Lambda)=9$}\label{case9}
Let $\Lambda $ be some dual strongly perfect lattice   so that 
$\Lambda \leq \RR ^{16}, \min(\Lambda ) \min (\Lambda ^* ) = 9 .$
Rescale the situation so that  
$$m:=\min (\Lambda )= 3/2 \mbox{ and } r:= \min (\Lambda ^*) = 6$$
and put $\Gamma := \Lambda ^*$. 
Then there are 
$a,b \in \{2,\ldots , 28 \} $ such that 
$$s:= | \Min (\Lambda ) | /2 = 2^7  a \mbox{ and } t:= |\Min(\Gamma  )|  = 2^7  b .$$
Then for all $\gamma , \gamma '  \in \Gamma $ the following numbers are integers: 
\begin{align*}
	(D4)(\gamma ):&\ 3a (\gamma , \gamma )^2,\\
	(D22)(\gamma,\gamma') :&\  a(2(\gamma,\gamma ')^2 + (\gamma,\gamma ) (\gamma' ,\gamma ') , \\
	\frac{1}{12}(D4-D2)(\gamma ):&\ \frac{a}{4} (\gamma ,\gamma ) ((\gamma ,\gamma ) -4 ) ,\\
	\frac{1}{6}(D13-D11)(\gamma',\gamma) :&\ \frac{a}{2} (\gamma ,\gamma ' ) ((\gamma ,\gamma ) -4 ) . 
\end{align*}

\begin{lemma}\label{N3}
If there is $\alpha \in \Min (\Gamma )$ and $x\in \Min (\Lambda )$ such that 
$(\alpha , x ) = 3$, then $\alpha = 2x $, $N_3(\alpha ) = \{ x\} $, 
$N_2(\alpha ) = \emptyset $ and $a=2$.
\end{lemma}

\begin{proof}
Clearly $\alpha = 2x$, so $x$ is uniquely determined by $\alpha $.
Assume that there is $y\in \Min (\Lambda )$ with $(y,\alpha ) = 2$. 
Then $(y,x) = \frac{1}{2} (y,\alpha ) = 1$ and $x-y \in \Lambda $ has norm
$(x-y,x-y) = 3-2=1 < 3/2$ a contradiction
to the fact that $\min (\Lambda ) = 3/2$. 
Therefore $N_2(\alpha ) = \emptyset $, $|N_3(\alpha ) | = 1$ and hence 
$$\frac{1}{12}(D4-D2)(\alpha ) = 6 = \frac{12 a}{4}  $$ 
implying $a=2$.
\end{proof}

\begin{lemma}\label{bounda}
Assume that $N_3(\alpha ) = \emptyset $. 
Then $a\leq 19$. 
\end{lemma} 

\begin{proof}
Then $|N_2(\alpha ) | = 3a$ and the set $\overline{N_2(\alpha )} := \{ \overline{x}:= x-\alpha /3 \mid 
x\in N_2(\alpha ) \} \subseteq \alpha ^{\perp } \cong \RR ^{15}$ 
satisfies 
$$(\overline{x} , \overline{x'} ) = (x-\alpha /3 , x'-\alpha /3) = (x,x' ) -2/3  \left\{
\begin{array}{ll} = 5/6  & x=x'\\ \leq  1/12 & x\neq x' \end{array} \right. $$ 
	so $\sqrt{6/5} \overline{N_2(\alpha )}  $ is a $[-1,1/10]$-spherical code in $S^{14}$.
By Lemma~\ref{SphCdeUppBnd} the cardinality
of such a code is upper bounded by $57=3\cdot 19$. 
\end{proof}

\begin{lemma}\label{9squarefree}
	If $a$ is squarefree then $\Gamma $ is an even lattice of level involving only
	the primes 2 and 3.
\end{lemma}
\begin{proof}
	$\frac{1}{12}(D4-D2)$ shows that $(\gamma ,\gamma ) \in 2\ZZ $ for all $\gamma \in \Gamma $.
	For the level, need to go through the possibilities for $b$.
	But $p^2$ does not divide $b$ for $p\geq 5$ so this is easy.
\end{proof}

\begin{corollary}
  $a$ is not squarefree.
\end{corollary}
\begin{proof}
  By Lemma~\ref{9squarefree} $\det(\Gamma)=2^a3^b$ for some nonnegative integers $a,b$
  and $\min(\Gamma^*)=3/2$. So by Lemma~\ref{minge3/2} $\Gamma$ is isomorphic to one of 
  $\Lambda_{16},\Gamma_{16},$ or $ O^*_{16}$, but none of them has Berge-Martinet invariant equal to 9.
  This concludes our proof.
\end{proof}

So we are left with the cases $a\in \{4,8,9,12,16,18 \} $. By symmetry we also conclude that 
$b\in \{ 4,8,9,12,16,18 \}$. By the modular form approach we can prove that

\begin{lemma}
  There is no dual strongly perfect lattice $\Lambda\subset\RR^{16}$ with $r(\Lambda)=9$ and
  $s(\Lambda)=2^7a$ and $s(\Lambda^*)=2^7b$ for some $a,b\in\{4,8,9,12,16,18 \}.$
\end{lemma}

In summary we have the following.
\begin{theorem} \label{r9}
  There is no dual strongly perfect lattice $\Lambda\subset\RR^{16}$ with $r(\Lambda)=9$.
\end{theorem}

\section{The case $r(\Lambda)=8$}\label{case8}
Throughout this section we assume that $\Lambda $ is a dual strongly perfect lattice of 
dimension $16$ with $r(\Lambda)= 8.$
Rescale $\Lambda$ so that $\min (\Lambda )= 2 \mbox{ and } \min(\Lambda ^*) = 4$.
Put $\Gamma :=\Lambda ^*$. By Theorem~\ref{TableOfValues} there are 
$a,b \in \{2,\ldots , 30 \} $ such that 
$$s:= | \Min (\Lambda ) | /2 = 2^3 3^2 a \mbox{ and } t:= |\Min(\Gamma  )| /2 = 2^3 3^2 b .$$
Then for all $\gamma , \gamma '  \in \Gamma $ the following numbers are integers: 
\begin{align*}
	(D4)(\gamma ):\ &3a (\gamma , \gamma )^2,\\
	(D22)(\gamma,\gamma') :\  &a(2(\gamma,\gamma ')^2 + (\gamma,\gamma ) (\gamma' ,\gamma ')), \\
	\frac{1}{12}(D4-D2)(\gamma ) :\ &\frac{a}{4} (\gamma ,\gamma ) ((\gamma ,\gamma ) -3 ) ,\\
	\frac{1}{6}(D13-D11)(\gamma',\gamma):\ &\frac{a}{2} (\gamma ,\gamma ' ) ((\gamma ,\gamma ) -3 ). 
\end{align*}

\begin{lemma}\label{squarefree}
If $a$ is squarefree then $(\gamma,\gamma ) \in \ZZ $  for all $\gamma \in \Gamma $ and 
$$\Gamma ^{(e)} : = \{ \gamma \in \Gamma \mid (\gamma , \gamma ) \in 2\ZZ \} \subset \Gamma ^* \cap \Gamma $$
is a sublattice of $\Gamma$ with $|\Gamma:\Gamma^{(e)}|\in\{1,2,4\} $. 
\end{lemma}
\begin{proof}
	$\frac{1}{12}(D4-D2)(\gamma)$ shows that $(\gamma ,\gamma ) \in \ZZ $ for all $\gamma \in \Gamma $.
If $(\gamma ,\gamma ) \in 2\ZZ $, then $\frac{1}{6}(D13-D11)$ 
implies that $(\gamma, \gamma ' ) \in \ZZ $ for all $\gamma ' \in \Gamma $, so 
$\Gamma ^{(e)} \subset \Gamma ^* \cap \Gamma $ is a sublattice of $\Gamma $.
\end{proof}

\begin{lemma}\label{odd}
If $a$ is odd then $a\in \{ 9,25 \} $. 
\end{lemma}
\begin{proof}
If $a$ is odd and squarefree then for $\alpha \in \Min (\Gamma )$ 
equation $\frac{1}{6}(D13-D11) $  shows that 
$\frac{a}{2} (\alpha ,\gamma ') \in \ZZ $ for all  $\gamma ' \in \Gamma $. 
This shows that $\frac{\alpha }{2} \in \Gamma ^* = \Lambda $ contradicting the
fact that $\min (\Lambda ) = 2 > 1 = (\frac{\alpha }{2}, \frac{\alpha }{2} ) $.
\end{proof}

\begin{corollary}\label{corodd}
	The argument above shows that $\frac{a}{2} \alpha \in \Gamma ^*$ 
	for all $\alpha \in \Min(\Gamma )$.
\end{corollary}

We now fix $\alpha \in \Min (\Gamma ) $ and consider the set 
$$N_2(\alpha ) := \{ x\in \Min (\Lambda ) \mid (\alpha , x ) = 2 \} $$
Then $|N_2(\alpha ) | = a $ and by \cite[Lemma 2.10]{dim12} 
we may write 
$$N_2(\alpha ) = E_1 \cup\ldots \cup E_k $$ where $E_i$ is minimal 
so that $\sum _{x\in E_i } x = \frac{|E_i|}{2} \alpha $ and $k$ is maximal.
Then 
$$\dim \langle N_2(\alpha ) \rangle = 1 + |N_2(\alpha ) | - k \mbox{ and } |E_i| \geq 2 \mbox{ for all } i .$$

\begin{lemma}\label{25} 
$a\neq 25$.
\end{lemma}
\begin{proof}
If $a=25$ then by the above 
$1+25-k \leq 16 $ implies that $k\geq 10 \geq 25/3 $. 
So there is some $i$ such that $|E_i| = 2$ which shows that $\alpha \in \Gamma ^*$. 
By Corollary \ref{corodd} we also have $\frac{25}{2} \alpha \in \Gamma ^*$ 
so in total $\frac{\alpha}{2} \in \Gamma * $ contradicting the fact that 
$\min (\Gamma ^* ) =2 $.
\end{proof}

So now we are left with the following cases:
$$a,b\in \{2,4,6,8,9,10,12,14,16,18,20,22,24,26,28,30\}.$$

\begin{lemma}\label{lemma:r=8_minimum}
  \begin{enumerate}[(i)]
    \item 
  If $a\in \{2, 4, 6, 8, 10, 12, 14, 20, 22, 24, 26, 28, 30 \}$
  then rescaling $\Gamma$ yields an even lattice of minimum $8$ (with dual minimum $1$). 
\item
  If $a\in\{ 9,18\}$ then rescaling $\Gamma $ yields an even lattice of minimum 24 (with dual minimum $1/3$). 
\item
	If $a= 16$ then rescaling $\Gamma $ yields an even lattice of minimum 16 (with dual minimum $1/2$). 
  \end{enumerate}
\end{lemma}

\begin{lemma}\label{r8a30}
  If $a=30$ then $\Gamma \cong \Lambda _{16}$.
\end{lemma}
\begin{proof}
	Assume that $a=30$. 
	Then $16\geq 1+30-k$ and $k\leq 15$ implies that $k=15$ and $|E_i| = 2$ for all $i$.
	So $N_2(\alpha ) = \{ x_1,\ldots , x_{15} \} \cup \{ \alpha -x_1,\ldots , \alpha - x_{15} \}  $
	and  
	$(x_i,x_j) =  1$ for all $i\neq j $. 
  Hence the lattice $L:=\langle N_2(\alpha)\rangle \subseteq \Lambda$.
  On the other hand, from Lemma~\ref{lemma:r=8_minimum} we know that $|\Lambda/L|$ has only the
  prime divisors $2$ and $3$. A complete search of the strongly perfect overlattices of $L$ with 
  minimum $2$ and whose determinant only have the prime divisors $2$ and $3$ shows that $\Lambda \cong \Lambda _{16}^{*}$
  and hence 
  $\Gamma = \Lambda ^* \cong \Lambda _{16}$.
\end{proof}

By the modular form approach, we can prove the following.
\begin{theorem}
  There is no dual strongly perfect lattice with 
  \begin{enumerate}
    \item $a,b\in \{2, 4, 6, 8, 10, 12, 14, 16, 20, 22, 24, 26, 28\}$ except for $a=b=28$;
    \item $a\in\{9,18\}$ and $b\in\{2, 4, 6, 8, 10, 12, 14, 20, 22, 24, 26, 28\}$
  \end{enumerate}
  and vice versa.
\end{theorem}

\begin{lemma}
  There is no dual strongly perfect lattice with $a=b=28$.
\end{lemma}
\begin{proof}
  By Lemma~\ref{lemma:r=8_minimum} we see that the level of $\Lambda$ divides $8$. 
  Then by the modular form approach we find that only the case $\det(\Lambda)=2^{-8}$ is possible.
  By the above $1+28-k\leq16$ implies that $k\ge13$. WLOG we have the following three possible cases:
  \begin{enumerate}[(i)]
    \item $k=13$, $|E_i|=2$ for $1\leq i\leq 11$ and $|E_{12}|=|E_{13}|=3$;
    \item $k=13$, $|E_i|=2$ for $1\leq i\leq 12$ and $|E_{13}|=4$;
    \item $k=14$, $|E_i|=2$ for $1\leq i\leq 14$.
  \end{enumerate}
  Case (i) can be easily excluded as the condition $|E_{12}|=3$ implies that $\alpha/2\in\Gamma^*$,
  which contradicts the fact that $\min(\Gamma^*)=2$.\\
  For Case (ii) we assume that $E_{i}=\{x_i,\alpha-x_i\}$ for $1\leq i\leq 12$, and
  $E_{13}=\{x_{13},x_{14},x_{15},x_{16}\}$. 
  So $(x_i,x_i)=2$ for $1\leq i\leq 16$ and $(x_i,x_j)=1$ for $1\leq i\leq 12$, $1\leq j\leq 16$ 
  and $i\neq j$.
  On the other hand, we know that $2\Lambda$ is even and $E_{13}$ is minimal so that 
  $\sum_{x\in E_{13}}=2\alpha$,
  hence $(x_i,x_j)\in\{0,\pm 1/4,\pm 1/2, \pm 3/4, -1\}$ for $13\leq i\ne j\leq 16$.
  A simple calculation shows that there is up to isomorphism only one possibility for
  the Gram matrix formed by vectors $x_{13},x_{14},x_{15},x_{16}$:
  \begin{align*}
    G=\begin{pmatrix}
      2 & 3/4&  3/4&  1/2\\
      3/4 &  2 &1/2 &3/4\\
      3/4& 1/2&   2& 3/4\\
      1/2& 3/4& 3/4&   2
    \end{pmatrix}.
  \end{align*}
  But the the norm of the vector $(\alpha-x_{14}-x_{15})$ is equal to $1$, contradicting the fact
  that $\min(\Gamma^*)=2$. This excludes Case (ii).\\
  For Case (iii) we assume that $E_{i}=\{x_i,\alpha-x_i\}$ for $1\leq i\leq 14$.
  So $(x_i,x_j)=1$ for $1\leq i\neq j\leq 14$. 
  Write $N_2(x_1)=\{\beta\in\Min(\Gamma)\mid(x_1,\beta)=2\}$. Similarly we can prove that
    $$N_2(x_1)=F_1\cup\dots\cup F_{14}$$
  where $F_i$ is minimal so that $\sum_{\beta\in F_i}=2x_1$. 
  Set that $F_i=\{\alpha_i,2x_1-\alpha_i\}$ for $1\leq i\leq 14$ where $\alpha_1=\alpha$.
  A computer search by MAGMA shows that there is up to isomorphism only one possibility
  for the lattice $L:=\langle x_1,\dots,x_{14},\alpha_1,\dots,\alpha_{14}\rangle$, and
  its determinant is equal to $4$.
  Then a complete search of the overlattices of $L$ with minimum $2$ and determinant $2^{-8}$
  shows that up to isomorphism there is only one such lattice and it is isometric to 
  $\Lambda_{16}^*$. This shows that Case (iii) is impossible.
\end{proof}

\begin{lemma}\label{a18b16}
	  There is no dual strongly perfect lattice with $a\in\{9,18\},b=16$ or $a=16,b\in\{9,18\}$.
  \end{lemma}
  \begin{proof}
	    By symmetry we may assume that $a\in\{9,18\}$ and $b=16$.
     If $a=9$ and $b=16$ then by the modular form approach we can prove that there is no
       such dual strongly perfect lattice.
  Now we assume that $a=18$ and $b=16$.
    We rescale $\Lambda$ such that $\min(\Lambda)=\frac{1}{3}$ and $\min(\Lambda^*)=24$.
      Set $\Gamma:=\Lambda^*$. In particular $\Gamma$ is even by Lemma~\ref{lemma:r=8_minimum}.
        Let $\alpha\in \Gamma$, and write $(\alpha, \alpha)=\frac{p}{q}$
	  with coprime integers $p$ and $q$. Then
	    \begin{align*}
		              (D4)(\alpha) &= {\frac{3 {p}^{2}}{2{q}^{2}}} \in\ZZ, \\
	              \frac{1}{6}(D13-D11)(\alpha,\beta) &= \frac{1}{2^3}(\alpha,\beta)((\beta,\beta)-18).
		      \end{align*}
       Hence $\Gamma^{(e)}:=\{\alpha\in\Gamma \mid (\alpha,\alpha) \in 4\ZZ\}$ is a sublattice of $\Gamma$
	  with $|\Gamma:\Gamma^{(e)}|\in\{1,2,4\}$.
    We apply the modular form approach to $\Gamma$ and $\Gamma^{(e)}$, and find that only the case
     $\det\Gamma=2^{46}3^2$ and $|\Gamma:\Gamma^{(e)}|=2$ is possible.
      Now from the linear restrictions~\eqref{LP} we find that $\Gamma$ contains at most 2426 vectors of
 norm $36$. On the other hand, as $\Lambda\subset(\Gamma^{(e)})^*$, we know that $(\Gamma^{(e)})^*$
    contains at least $2\cdot72\cdot18$ vectors of norm $\frac{1}{3}$.
     Also as $\Min(\Gamma)=\Min(\Gamma^{(e)})$, $\Gamma^{(e)}$ is also strongly perfect, so
       $\min((\Gamma^{(e)})^*)\ge 1/4$. Now from the linear restrictions~\eqref{LP} and the condition that
  $(\Gamma^{(e)})^*$ contains at least $2\cdot72\cdot18$ vectors of norm $\frac{1}{3}$, we compute that
   $\Gamma^{(e)}$ contains at least $6172$ vectors of norm $36$, which is a contradiction.
     This concludes our proof.
    \end{proof}

\begin{lemma}
 There is no dual strongly perfect lattice with $a,b\in\{9,18\}$.
  \end{lemma}
  \begin{proof}
   By Remark~\ref{shortTableOfValues} we see that the case $a=b=9$ is impossible.
  By symmetry we may assume that $a=18$ and $b\in\{9,18\}$.
   We rescale $\Lambda$ such that $\min(\Lambda)=\frac{1}{3}$ and $\min(\Lambda^*)=24$.
   Set $\Gamma:=\Lambda^*$. As in Lemma \ref{a18b16} we see that 
     $\Gamma $ is an even lattice and 
     $\Gamma^{(e)}:=\{\alpha\in\Gamma \mid (\alpha,\alpha) \in 4\ZZ \} $ is 
     a sublattice of $\Gamma$ with $|\Gamma:\Gamma^{(e)}|\in\{1,2,4\}$ (see \cite[Lemma 2.8]{dim12}). 
     Similarly we can prove that $L:=\sqrt{72} \Lambda $ is even 
and hence  the even level of $\Gamma$ divides $72$.
  If the even level of $\Gamma$ divides $36$ or $24$, then we apply the modular form technique
 to the lattice $\Gamma$, and the computation shows that
 there does not exist such a lattice. 
 So in the following we assume that the level of $\Gamma$
     is equal to $72$. 
     Then $\Gamma^{(e)}$ is a proper sublattice of $\Gamma $.
 As $\Min(\Gamma)=\Min(\Gamma^{(e)})$, $\Gamma^{(e)}$ is also strongly perfect, so
$\min((\Gamma^{(e)})^*)\ge 1/4$. 
In total we find 1508 possible genus symbols for the lattice $\Gamma$.
   We apply the modular form technique to $\Gamma$ if 
   $\det(\Gamma ) \not\in \{ 2^{18}3^{20},2^{30}3^{12},2^{24}3^{16} \} $ and to its even sublattice $\Gamma^{(e)}$
   otherwise.
It turns out that none of the 1508 genus symbols is possible. This concludes our proof.
   \end{proof}

\bibliography{SPL}
\bibliographystyle{plain}

\end{document}